\def\ZZ         {{\mathbb Z}} 
  
\def\CC         {{\mathbb C}}

\def\PP         {{\mathbb P}}
\def\NN         {{\mathbb N}}
\def\ZZ         {{\mathbb Z}}
 
\def\X          {{\cal X}} 
\def\A         {{\cal A}}																																																	

\def\C         {{\cal C}}

\def\CL          {{\mathcal L}}

\def\CO         {{\mathcal O}}

\def\S           {{\cal S}}

\def\X           {{\cal X}}

\def\cal        {\mathcal}

\documentclass{amsart}
\usepackage[colorlinks=true,allcolors=blue!80!black]{hyperref}
\usepackage{tikz, tikz-cd}
\usetikzlibrary{matrix, snakes, patterns, shadows.blur, backgrounds}
\usetikzlibrary{decorations.markings}
\usepackage{amssymb}
\usepackage{verbatim}
\usepackage{eucal}
\usepackage{enumerate}
\usepackage{tikz}
\usetikzlibrary{decorations.markings,decorations.pathmorphing,cd}
\usetikzlibrary{arrows}
\usetikzlibrary{calc}
\usepackage{pgfplots}
\pgfplotsset{every axis/.append style={
                    axis x line=middle,    
                    axis y line=middle,    
                    axis line style={-,color=blue}, 
                    xlabel={$x$},          
                    ylabel={$y$},          
            }}

\usepackage{amsthm}
\usepackage{amssymb}
\usepackage{enumerate}
\usepackage{graphicx}
\usepackage{hyperref}
\newtheorem{thm}{Theorem}[section]
\newtheorem{theorem}{Theorem}[section]

\newtheorem{prop}[theorem]{Proposition}

\newtheorem{cor}[theorem]{Corollary}
\theoremstyle{definition}
\newtheorem{rem}[theorem]{Remark}
\newtheorem{dfn}[theorem]{Definition}

\newtheorem{example}[theorem]{Example}

\theoremstyle{remark}

\title[Curves with restricted classes of components]
{Fundamental groups of the complements to reducible 
curves on smooth surfaces with restricted classes of components.} 
\author{A.Libgober}

\address{Department of Mathematics\\
University of Illinois\\
Chicago, IL 60607}
\email{libgober@uic.edu}
\thanks{}
\pgfplotsset{compat=1.18}
\begin{document}
\begin{abstract}
We classify equisingular families of 
pencils of curves on smooth simply connected projective surfaces whose dual varieties have degree less than 6
and  containing a sufficiently large number of members whose irreducible components are either hyperplane sections or irreducible components of 
hyperplane sections. 
As a consequence, we describe the fundamental groups 
of complements to curves on such surfaces whose irreducible components satisfy the above condition and whose fundamental groups admit
an essential surjection onto a free group of rank greater than 6. 
\end{abstract}
\maketitle 

\section{Introduction.} Let $X$ be a smooth simply connected projective surface. We are interested 
 in reduced curves $D$ on $X$ such that $\pi_1(X\setminus D)$ admits a surjection
onto a non-abelian free group. Such pairs $(X,D)$ were considered recently 
in \cite{cogome}, \cite{handbook}, suggesting a path to their classification up to equisingular deformations, assuming 
that the rank of the free quotient of $\pi_1(X\setminus D)$ is sufficiently large and if 
 one restricts the numerical classes of irreducible components of $D$ by requiring that they belong 
 to a fixed subset of the Neron-Severi group.
  This can be viewed as a further refinement of the dichotomy between "small and large" quasi-projective groups pointed out in \cite{arapurafibred} 
 (cf. also \cite{catanesefibred}). 
 The goal of this paper is to describe such a classification when $X \subset \PP^N$ has  the degree of the dual variety less than 6 and calculate the fundamental groups $\pi_1(X\setminus D)$ 
 assuming that $\pi_1(X\setminus D)$ admits an essential and transversal surjection (cf. Definition \ref{threshold}) onto a free group of rank greater than 6 and 
 that irreducible components of $D$ are either hyperplane sections or are their irreducible components.
  
The key ingredient in the study of $\pi_1(X\setminus D)$ having free quotients of rank at least 2
is the fundamental result on quasiprojective manifolds admitting families of rank one 
local systems  with non-vanishing cohomology (cf.\cite{arapura}). Such a family on $X\setminus D$
 must be a pullback of a family of local systems on the complement to a finite set of points in a curve, which since $X$ is simply connected must be 
 $\PP^1$, via 
a holomorphic map corresponding to a rational pencil on $X$. Moreover, as is shown in Proposition \ref{pencils} below, with appropriate conditions on the surjection of $\pi_1(X\setminus D)$ onto a free group, the curve $D$ is a union of several reduced members of this pencil. 
 This replaces the problem of classifying of curves $D$ with $\pi_1(X\setminus D)$ admitting 
 free quotients having rank greater than one, 
by the problem of classification of the pencils on $X$.
 
An early result, suggesting interesting finiteness properties for pencils  assuming that the  
classes of all irreducible components of $D$ belong to a fixed set $\nabla\subset NS(X)$, 
 concerns the case when $X=\PP^2$ and the curve $D$ is an arrangement of  lines i.e. a curve
 with all irreducible components restricted to having the class $[1]\in NS(\PP^2)=\ZZ$. It was shown 
in \cite{libyuz}, that a pencil of plane curves having more than 5 members that are unions 
of lines must be a pencil of concurrent lines (later, the technical assumptions in \cite{libyuz} were weakened and the threshold 5 was sharpened to 4 in  \cite{stipins}, \cite{yuz4}, \cite{falkyuz}). 
Related works, considering the number of reducible members in the pencil, though 
without restricting the classes of components, are \cite{vistoli},\cite{ruppert}. 

In \cite{cogome} it was shown that if $r \in \NN$ is above a certain threshold $M_2(X,\nabla)$, depending on the surface and a set $\nabla \subset NS(X)$, then a pencil of curves on $X$ with $r+1 \ge M_2(X,\nabla)$ members having component
in $\nabla$ must be a pencil of curves already having all its members in a class in $\nabla$. In the same work, some thresholds were calculated, and in particular it was shown that a pencil of plane curves containing more than 
6 members with irreducible components being either conics or lines (i.e. $D$ is a conic-line arrangement) then this 
pencil is a pencil of conics (or lines). 

In \cite{handbook} it was observed that the number of pencils is related to complexity of the discriminants of the linear systems
in $\nabla$. Recall that the discriminant in $\PP(H^0(X,L))$ where $L$ is a line bundle, is the subvariety (possibly empty)
consisting of singular divisors in this complete linear system. If $\vert L\vert $ is the linear system of hyperplane sections, then the discriminant is 
what is classically known as the dual variety of $X$ (cf. \cite{kleiman} for classical overview of dual varieties).  
Main results below consider $\nabla \subset NS(X)$ such that  all line bundles $\CO(\delta)$,  where the class of $\delta$ is in  $\nabla$,  
have a small degree of discriminant and calculate the groups $\pi_1(X\setminus D)$ with irreducible components of $D$ having classes in $\delta \in \nabla$. 
The condition, that $X$ admits an ample linear system with the degree of discriminant 
less than 6 is very restrictive. We show that 
the classical results (cf. \cite{fujita0}, \cite{fujita1}, \cite{lanteri} and references therein) imply that such polarized surfaces are either $(\PP^2,\CO(1)),(\PP^2,\CO(2)),(V_2,\CO(1))$ or  
scrolls of degree less than 6 (cf. Cor. \ref{atmost6}). At this point, we note that in this paper we consider only the case when pencils satisfy additional
technical conditions (cf. Definition \ref{threshold}) which assure that the pencils are sufficiently generic. A complete linear system on a surface, even with a 
small degree of discriminant and satisfying these technical conditions, may admits several (equisingular) types of pencils. We show that on the polarized surfaces with degree of discriminant less than 6 there are 17 families of pencils with constant equisingularity type of their singular members (cf. Section \ref{finalsection} and the Table therein). 

This classification of pencils in linear systems with small degree of discriminant
has as an immediate corollary the classification of the fundamental groups with free quotients of large rank.  The geometric conditions on
the pencils that we consider can be restated in group-theoretic terms as conditions of the surjection (cf. Proposition \ref{pencils} and Remark \ref{arrangconnection}) 
and we obtain classification for pairs $(X,D)$ such 
that $D$ has components moving in linear systems with the degree of discriminant less than 6 and have surjection 
onto a free group of rank $r$ greater than 6 which is the largest of the thresholds for such pairs $(X,\nabla)$. In 
particular any group, having surjection on a free group of rank greater than $6$ and satisfying the conditions that we call  
being "essential and transversal" (cf. Definition \ref{threshold}),   
is a product $F_k\times F_r, 1 \le k \le 4$, except for plane conic-line arrangements in which case the presentations of corresponding groups obtained 
explicitly (cf. Theorem \ref{presentationquadrics}). These results can be summarized as follows:

\begin{theorem} If the fundamental group of the complement to a curve on a  polarized simply connected surface 
admits surjection onto a free group of rank $r$ greater than 2 that is essential and transversal then the curve is a union of reduced members of a pencil 
and intersect transversally at the base points.
If the pencil lies in a very ample linear system, the degree of discriminant of polarization is less than 6 and if $r > 6$, then 
the surface, the pencil with only reduced members and corresponding fundamental group of the complement to the curve are listed in Table 1
in the final section.
\end{theorem}

A more detailed account of the results is as follows. 
We begin by recalling the key definitions and results, especially
those in \cite{cogome}. 

In section 3 we improve the upper bound on the thresholds i.e. the bound on the maximal number of members 
in pencils, with irreducible components in a fixed saturated set $\nabla \subset NS(X)$, that can appear in linear systems 
 that are outside of $\nabla$, 
 in the case of pencils of curves on scrolls. 
We show that for a polarized Hirzebruch surface $F_e,E+kF, k>e$, i.e. a scroll $S_{k-e,k}\subset \PP^{2k-e+1}$, 
an essential and transversal pencil having more than $max(2(k+1)-e,7)$ reducible members with irreducible components being 
in saturated system $\nabla_k=\{E,F, E+lF, e \le l \le k)$ (i.e. the hyperplane sections of the scroll and their possible irreducible components), 
must be a pencil of curves with classes in $\nabla$. 
In fact, for pencils in linear systems $aE+bF$ with $a,b$
sufficiently large, the number of singular members is at most 6. It is an interesting open question whether there are pencils on scrolls, having 7 reducible 
members with components in $\nabla$ but in section 4, for a quadric $W$ in $\PP^3$, we describe a pencil of quadric surfaces inducing pencil 
with $6$ members split into a union of two hyperplane sections.

In section 5 we find all polarized surfaces with the degree of the dual variety at most 12 and identify those for which this degree  is at most 6. This is a classical material, though we did not find direct references in the literature, and results follow from numerous 
works on surfaces of small degree (our prime references are works by T.Fujita and A.Lanteri: \cite{fujita0},\cite{fujita1},\cite{lanteri}). 
In this section we also describe the stratification of the dual varieties of scrolls. The study of these questions goes back 
to A.Cayley and J.Sylvester and more recently, in much wider generality, to I.Gelfand, M. Kapranov and A. Zelevinski but we did not find 
specific references, in particular enumeration of strata, and their explicit description (cf. Theorem \ref{stratification}). 
However, the work \cite{ciliberto} does contain some overlaps with this section (cf. Prop. 1.5, ibid). 

In the final sections, we calculate the fundamental groups of the complements for all essential and transversal pencils on 
polarized surfaces with degree of discriminant being at most 6. 
 
The main step in our classification of groups admitting free quotients of large rank, is the case of arrangements of curves of surfaces 
associated with pencils of curves 
in the following way. Let $\vert L\vert$ be a complete linear system on a polarized smooth surface $X$ and let $\S_1,\cdot  \S_k$ be the full collection 
of strata of the dual variety of $X$ (i.e. the discriminant of $\vert L \vert$) which admit a secant line in the dual space of $\vert L\vert$. 
\begin{dfn}\label{generalceva} The arrangement $\A^N_S(\S_1,\cdots, \S_k)$ 
corresponding to a polarized surface $X,L$ and a collection of strata $\S_1,\cdots,  \S_k$ of the dual variety of $X$ is a reducible curve on $X$ formed 
by a set of $k$ singular members of a pencil $P$ intersecting the dual variety only at the points in the strata  $\S_i, i=1,\cdots, k$  and $N-k$ smooth members of $P$.
\end{dfn}
The arrangements $\A^N_S(\S_1\cdots \S_k)$ generalize several classical configurations of lines (cf. \cite{dolg}). For example, the Ceva arrangement $(6_2,4_3)$
is the arrangement of lines formed by singular members of a generic pencil of plane conics i.e. is $\A(S,S,S)$ where $S$ is 
the 4-dimensional stratum of the discriminant of the linear system of plane quadrics, corresponding to those that are reducible and reduced. The Hesse arrangement $(12_3,9_4)$ corresponds to the pencil of plane cubics, intersecting only the strata of reducible cubics that are union of three non-concurrent lines i.e.
is $\A(T,T,T,T)$ where $T$ is the latter stratum of the discriminant of plane cubics, etc.

For polarized surfaces with the degree of discriminant less thant 6 with saturated system described in section 5, 
a pencil having more than 7 members with irreducible components in $\nabla$ must be a pencil of curves in $\nabla$ and 
 the fundamental groups of the complement to union $N>7$ members of such
pencil  are quotients of the fundamental groups of the complement to union of the arrangements $\A^N_S(\S_1\cdots \S_k)$ and several 
smooth members of the pencil. As a corollary we obtain the 
following:
\begin{cor} If fundamental group of the complement to a curve on a polarized surface with the degree of discriminant less thant 6 
admits an essential and transversal surjection onto a free group of rank greater than $6$, then this group is one of the groups 
in the last column in the table in section  \ref{finalsection}.
\end{cor}

The study of the pencils here is related to the study of Lefschetz pencils in symplectic geometry, especially in the case of scrolls 
(cf. \cite{baykur}) but the pencils we consider have singular members with several ordinary quadratic singularities.
 Several of results of this paper can be sharpened. It would be interesting to continue the enumeration of extremal group (cf. Definition \ref{threshold}), 
 i.e. the fundamental groups admitting free quotient of the rank above the threshold (we have limited ourself to the case of discriminant less than 6 
 to keep length from being excessive, though their number grows extremely rapidly after we move into the range where the degree of discriminant 
is 12 and higher (i.e. plane cubics, hyperplane sections of del Pezzo surfaces and so on). The outstanding question 
is if one can detect any interesting pattern in the fundamental groups of the complement beyond this bound 12. 

Finally, I want to thank I. Dolgachev for the reference to the classical book by G.Salmon (cf. \cite{salmon} and footnote in the proof of Corollary \ref{smalldiscrimlist}) and F.Catanese,  for his comments on the earlier works on the relation between surjections of the quasi-projective groups 
onto groups of hyperbolic curves and pencils.

\section{Preliminaries: Quasi-projective and affine fundamental groups with free quotients} 

In this section we recall the connection between the curves for which the fundamental group of the complement admits 
surjection onto $F_r, r>1$ and pencils of curves as well as some results from \cite{cogome}. Throughout the paper, 
$X$ is a smooth projective simply connected surface and $D \subset X$ is a reduced curve. Note that existence of surjection onto $F_r, r>1$
 implies that the curve must be 
reducible as follows from the following relation (a consequence of the exact sequence of the pair $(X,X\setminus D)$ cf. \cite{handbook}, Section 3): 
\begin{equation} H_1(X\setminus D,\ZZ)=\{[D_i]\}/Im H_2(X,\ZZ)
\end{equation} 
where $\{[D_i]\}$ is free abelian group with generators corresponding to irreducible components $D_i$ of $D$, and the map of 
$H_2(X,\ZZ)$ to the free abelian group generated by the classes of irreducible components, takes a class $\alpha$ to 
$(\cdots, (\alpha \cdot C_i),\cdots) \in \{[C_i]\}$ (in particular, if $D$ is irreducible then $H_1(X\setminus C,\ZZ)$ is cyclic).

\begin{dfn} Let $C\subset X$ be an irreducible curve on a complex surface $X$. A meridian of $C$ is 
a loop in $X\setminus C$ which is the oriented boundary of a small 2-disk in $X$ transversal to $C$ at  a smooth point.
Sometimes we also use the term "meridian" referring to {\it the conjugacy class} of the element of the fundamental group $\pi_1(X\setminus C,p), p \in X\setminus C$ represented by a loop consisting of a path in $X\setminus C$ connecting $p$ with a point of the meridian. 
Meridional rank $m(D)$ of a curve $D$ on $X$ with irreducible components $D_i$ is  the minimal number of meridians of irreducible components 
generating $\pi_1(X\setminus D,p)$. 
\end{dfn}

\begin{rem} Conjugacy of all meridians of each irreducible component follows from the connectivity of the complement in an irreducible curve $C$ to the set of singular points. 
In particular, we will use additional structure on $\pi_1(X\setminus D), D=\bigcup_1^N D_i$ (where $D_i$ are irreducible) given by  the set of 
$N$ conjugacy classes of $\pi_1(X\setminus D)$ of the meridians of irreducible components of $D$.
\end{rem}


The following Proposition makes precise the correspondence between surjections onto free groups of rank greater 
than 1 and pencils of curves (for much earlier relation of such type cf. \cite{ingrid}, \cite{catanese2000}, \cite{catanesefibred}) as well as earlier works mentioned there.

\begin{prop}\label{pencils} Let $X$ be a smooth simply connected projective surface and let $D$ be a reduced reducible curve  
such that  $\pi_1(X\setminus D)$ admits a surjection $\phi$
onto a free group $F_r, r>1$.
Assume the following: 

(a) the surjection $\phi$
maps a meridian of each irreducible component of $D$ to one of $r$  fixed generators $x_i$ of $F_r$ or $(x_1\cdots x_r)^{-1}$, and 

(b) at each intersection point of components of $D$ the singularity of $D$ is an ordinary multiple point (transversal intersection of smooth branches). 

Then there exists a pencil on $X$ such that $D$ is a union of several reduced members of this pencil. Moreover, the pull back of this pencil 
to the blow up of $X$, once at each base point of this pencil, is base point free.
\end{prop} 

\begin{rem}\label{arrangconnection} 
  Expanding on the Introduction, one can spell out the connection between the last Proposition and the Proposition 7.3 in \cite{libyuz}. Let $\A$ be an arrangement of $n(r+1)$ lines such that a) the incidence matrix defined in \cite{libyuz} consists of $(r+1)$ blocks each corresponding 
to a group of $n$ lines in general position,  b) any point of the arrangement having multiplicity greater than 2 contains one  line from each of $r+1$ groups
and c) each line contains $n$ points of of multiplicity greater than 2.
It was shown in \cite{libyuz} that this condition implies the existence of the pencil, having several 
members, each being a union of lines, giving a surjection of $\pi_1(\PP^2\setminus \A)$ onto a free group. One immediately sees that conditions (a),(b) in Prop. \ref{pencils} are equivalent to conditions a),b),c) in Prop. 7.3 in \cite{libyuz}.
Conversely, Proposition \ref{pencils} implies that surjection onto non-abelian free group induces a pencil on $\PP^2$ with $\A$ being the union of its reducible members. We obtain a partition of $\A$ into groups, satisfying the conditions a),b), c),
each consisting of the irreducible components of a particular member of this pencil.  

The conclusion of the Proposition 7.3 in \cite{libyuz}, using improved in \cite{stipins} and \cite{yuz4} inequlity $r+1 \le 4$, can be restated in terms of the fundamental groups: if $\pi_1(\PP^2\setminus \A)$ has surjection onto $F_r, r>3$ then the subarrangement of lines which meridians have nontrivial image in $F_r$ is a union of concurrent lines, since the pencil corresponding to such surjection must be a pencil of curves of degree 1.  Moreover, if $r=3$,assuming well known conjecture that the only arrangement as in Proposition 7.3 in \cite{libyuz}
is Hesse arrangement, one could conclude that if a fundamental group of an arrangement has surjection onto $F_3$ satisfying (a),(b) in Prop. \ref{pencils}, then this group admits a surjection onto 
the Hesse group which is a semidirect product of $F_3$ and $F_{10}$.  
\end{rem}

\begin{proof} Note that pull back via $f$ of characters of $F_r$ induces a subtorus in ${\rm Char}(\pi_1(X\setminus D)$ consisting of characters $\chi$, such that 
the cohomology of the corresponding local system on $X\setminus D$ are not vanishing. 
Recall that a map $\X \rightarrow \C$ of quasiprojective manifold $\X$ onto a curve  $\C$ is called {\it admissible} (cf. \cite{arapura}) if there is an extension $\hat \X \rightarrow \hat \C$ with connected fibers onto a smooth compactification of $\C$.  
It follows from [1], Prop.1.7 and Cor.1.9 that there is an admissible map of $f_{X\setminus D}: X\setminus D \rightarrow \PP^1\setminus D'$ where $D'$ is a finite subset on $\PP^1$ such that $f_{X\setminus D}^*(H^1(\PP^1\setminus D'))=f^*(Char F_r)$ (in particular the target of extension is $\PP^1$).
 Let $\hat f_X: \hat X\rightarrow \PP^1$ be an extension.
This map produces the pencil in complete linear system $f^*_{\hat X}(\CO_{\PP^1}(1))$ and if it has fixed components which union is $H$, then the corresponding pencil in the complete linear system corresponding to $f^*_{\hat X}(\CO_{\PP^1}(1))\otimes \CO(-H)$ is pull back of the pencil on $X$ having a finite set of base points $B$ with corresponding 
rational map $f_X: X\rightarrow \PP^1$ being regular on $X\setminus B$. Due to assumption (a), restriction of $f_X$ to a components of $D$ cannot be a dominating
map onto $\PP^1$, since the meridian of such components will have trivial image in $F_r$, i.e. all components of $D$ are components of  members of the pencil on $X$. Also, since a map of a regular neighborhood of fiber of its morphism onto a disk $\Delta$ having multiplicity $m$ takes the meridian to 
the $m$-th power of positive generator of $\pi_1(\Delta^*)$,  we obtain the reducibility assertion. Moreover, (b) implies that $\hat X$ is the 
blow up of $X$ at $B$. 
\end{proof} 

\begin{rem} The condition (a), in terminology of \cite{abcov} is equivalent to the condition that the component $f^*_{X\setminus D}({\rm Char} F_r)$ of the characteristic variety of $X\setminus D$ 
is essential but (a) states it directly in terms of the fundamental group. 
Specifically, existence of  surjection of $\pi_1(X\setminus D) \rightarrow F_r$ implies that for any curve $D' \subset X$ one has also surjection 
$\pi_1(X\setminus D\cup D') \rightarrow \pi_1(X\setminus D) \rightarrow F_r$. If the components of $D'$ satisfy the conditions  (a) and (b)  then 
this composition corresponds to composition of maps $X\setminus D\cup D' \rightarrow X\setminus D \rightarrow \PP^1 \setminus [r+1]$. 
We always assume that the pencil is essential (relative to the curve $D$) in the sense that the map $X\setminus D \rightarrow \PP^1$ is not a restriction of the map of the complement to a curve with irreducible components being a proper subset of $D$ \footnote{Essential pencils correspond via correspondence to referred to in Prop. \ref{pencils} to essential components of characteristic varieties cf. Definition \ref{threshold}}
.  For example for Hesse arrangement, we would not allow to use the map which is 
a linear projection from one of its quadruple points. Hence the condition (a) simultaneously assures that the corresponding map of $X$ has only reduced fibers and for no union $D'$ of components of $D$ the map $X\setminus D\rightarrow \PP^1\setminus [r+1]$ inducing surjection of the $F_r$ is restriction of the
map of $X\setminus D'$ inducing the same surjection.
\end{rem}

 Next we recall the results of showing the finiteness of pencils (up to equisingular deformations cf. Remark \ref{toptypesofpencils} below) of curves 
with classes of irreducible components in a saturated set and with fundamental groups 
of the complement having a free quotient of sufficiently large rank.
 Recall that a collection of effective divisors $\nabla \in Eff(X)$ is called saturated if for any  pair of effective divisors $D_1,D_2$ such that $D_1+D_2 \in \nabla$  one also has $D_1\in \nabla$ and $D_2 \in \nabla$. 
 The key case is when $\nabla$ is the smallest saturated set consisting of classes of divisors
  that are irreducible components of curves in a very ample linear system $D$ and when 
 generic element of a pencil has class in $D$. We will assume that this is the case. 
In the case, of $\PP^2$, the set of curves in $H^0(\PP^2,\CO(d)), 1 \le d \le k$ for a fixed $k$ is saturated. For a scroll $F_e$, that is projectivization  
$\PP_{\PP^1}(\CO \oplus \CO(-e))$ (cf. \cite{hart}, Ch. 5, Section 2)
the classes containing irreducible curves are $aE+bF$ where $(a,b)=(1,0),(0,1)$ and  
$(a,b), a>0,b \ge ae$  (cf. \cite{hart}, Ch.V, Cor. 2.18). It particular it follows that the set $\nabla(D_k)=(E+lf, e\le l \le k, E,F)$ is saturated. 
 
\begin{thm}\label{paperwithjose} Let $X$ be a smooth simply connected
  projective surface. Let
$\nabla \subset NS(X)$ be a saturated subset of the effective cone. 
Let $D$ be a reduced curve having classes of its irreducible
components in $\nabla$ and satisfying conditions (a) and (b) in Prop. \ref{pencils}.
 Then there is a constant $M_2(X,\nabla)$ such that
if $\pi_1(X\setminus D)$ admits an essential surjection onto a free group $F_r$, 
where $r$ is any integer satisfying $r >M_2(X,\nabla)$, 
then irreducible components $D$ can be partitioned into $r+1$ groups with the following property.
The union   
 of irreducible components belonging to each group form a member of a pencil with generic element having 
 class  $\delta \in \nabla$. 

Moreover, there is a constant $M_1(X,\nabla) < 10$ such that for $r
>M_1(X,\nabla)$ there is only finite number
$N(X,r,\nabla)$ of isotopy classes of curves $D$
with components from $\nabla$ and 
admitting a surjection $\pi_1(X\setminus D)\rightarrow F_r$ but not composed of a pencil in
$H^0(X,{\cal O}(D))$ where $D$ is a divisor in a class  $\delta\in
\nabla$.
\end{thm} 

\begin{rem}\label{toptypesofpencils} Consider a complete linear system $\PP(H^0(X,\CL))$ of sections of a line bundle $\CL$. It admits an
equisingular stratification with finitely many strata formed by Zariski-open sets consisting of divisors with the same multiplicities 
and topological type of irreducible components as well as the local 
types of singularities (cf. \cite{handbook}).  Each such stratum has at most finitely many connected 
components. Two pencils of curves in this linear system, i.e. two lines $P_1,P_2$ in $\PP(H^0(X,\CL))$ are 
called equisingularly isotopic if there is a continuous deformation $P_t $ in the Grassmanian of lines in  $\PP(H^0(X,\CL))$ 
parametrized by a disk, such that for all $t$, 
the set of strata that $P_t$ intersects and the number of such intersections are constant. 
From the finiteness of the number of strata in  $\PP(H^0(X,\CL))$ follows the finiteness of isotopy classes of pencils.
Also, a consequence of the Thom isotopy theorem is that two curves $D_1,D_2$, each formed by the unions of members of two isotopic pencils
i.e. $D_1=D_1^i, D_2=D_2^i, i=1,\cdots k$ such that $D_1^i$ and $D_2^i$   
belong to the same stratum for any $i$, have homeomorphic complements $X\setminus D_j, j=1,2$. Hence there are only finitely many fundamental groups of the complements
of curves which are unions of bounded number of members of pencils in $\PP(H^0(X,\CL))$, each member having irreducible components in $\nabla$. 
This gives explicit construction and enumeration of arrangement of curves corresponding to a saturated subset $\nabla$. 
Moreover, any such arrangement is one of the arrangements $\A^N(\S_1,\cdots \S_k)$ introduced in Definition \ref{generalceva}.



An interesting problem is to find sufficient conditions for a collection of strata admitting a collinear points.  
We will see below that there are many interesting examples on scrolls and saturated sets $\nabla$ admitting arrangements 
with components in a proper subset of $\nabla'$.
\end{rem}

\begin{dfn}\label{threshold}  Let $X$ be a simply connected projective surface and let $\nabla$ be a saturated family of effective divisors.

1. Let $D$ be a curve with classes of irreducible components in $\nabla$. 
A surjection $\pi_1(X\setminus D) \rightarrow F_r$ satisfying the condition (a) (resp. (b)) in Prop. \ref{pencils}
is called {\it essential} \footnote{this terminology is consistent with the one used in \cite{handbook}  Def. 4.14 .} (resp. {\it transversal}).
In agreement with this terminology, it is convenient to refer to a pencil with reduced members $C_1,\cdots C_N$, as essential (relative 
to these members) and {\it transversal} if distinct members intersect transversally at the base points.

2. The smallest integer $M_2(X,\nabla)$ such that for all $r>M_2(X,\nabla)$, a curve satisfying the condition of the theorem  $\ref{paperwithjose}$ admitting a surjection on a free 
group $F_r$  must be a curve which is a union of $r+1$ members of a pencil all irreducible members being the curves in $\nabla$ is called a threshold. 

3. A group $G$ is called $r$-extremal for $(X,\nabla)$ if 
 $r>M_2(X,\nabla)$ 
 there exist a pencil of curves in linear system
of divisors in a class $\delta \in \nabla$, a curve $D$ which irreducible components are members 
of this pencil, and an essential and transversal surjection $\pi_1(X\setminus D)\rightarrow F_r$ 
such that 
 $$G=\pi_1(X\setminus D)$$
\end{dfn}

\begin{cor}\label{finitecollectionsgroups}Let $(X,\nabla)$ be as in Theorem \ref{paperwithjose} and Let $r>M_2(X,\nabla), t \in \NN$. There is  an integer $s(X,\nabla,r)$ and 
a finite collection of finitely generated and finitely presented groups $Q^i_r(X,\nabla), i=1,\cdots, s(X,\nabla,r)$  such that if $\pi_1(X\setminus D)$
admits an essential surjection onto a free group $F_r$ then there exist $1\le i \le s(X,\nabla, r)$ such $\pi_1(X\setminus C)=Q_r^i(X,\nabla)$.
\end{cor}

In the next sections we will calculate the thresholds for polarized simply connected surfaces with the degree of discriminant
in the polarization not exceeding 12. In the Section \ref{smalldegree} we will show that those surfaces are precisely those of maximal degree 
(i.e. Fujita's $\Delta$-genus being zero).

\section{Bounds on the number of reducible fibers of pencils of scrolls 
with components from prescribed classes.}\label{boundsonnumber}

In this section we sharpen the results from \cite{cogome} in the cases 
scrolls and saturated sets consisting of irreducible components of all possible hyperplane sections.
Recall (cf. \cite{hart}), Ch. 5) that a scroll 
 is a projectivization of rank 2 bundle on $\PP^1$ embedded in projective space so that
the fibers of $\PP^1$-bundle over $\PP^1$ are lines. One can 
assume that this rank 2  bundle is $\CO_{\PP^1}\oplus \CO_{\PP^1}(-e), e \ge 0$ and denote such a surface as $F_e$. One has an 
identification $Pic(F_e)=\ZZ^2$ with 
generators being the class $E$ of a section such that $(E,E)=-e$ and class $F$ of a fiber. A divisor $aE+bF$ is ample (or very ample) iff $a>0, b>ae$. A  complete
linear system takes $F_e$ to a scroll iff $a=1$ and the map into  $\PP(H^0(F_e,\CO_{F_e}(D_k)), D_k=E+kF$ has as image the join in $\PP^{2k-e+1}$ of rational normal curves of degrees $k-e$ and $k$.  The image of $E$ (the directrix cf. \cite{GH}) is the only rational normal curve on the scroll having negative self-intersection. We denote such a scroll as 
$S_{k-e,k} \subset \PP^{2k-e+1}$ (cf. \cite{harris}). The rulings are lines that are the images of the fibers of $F_e\rightarrow \PP^1$.  The canonical class of $F_e$ is 
$-2E-(2+e)F$ and the class of tautological bundle $\CO_{\PP(\CO_{\PP^1} \oplus \CO_{\PP^1}(-e)}(1)$ is $E$. 

For a fixed $k$, the set of divisors $\nabla(D_k)=\{E,F,E+lF, e\le l \le k\}$ is saturated. The next theorem gives a bound on
the number of singular members in a pencil on $S_{k-e,k}$  with classes of irreducible components restricted to set $\nabla(D_k)$.
\begin{thm}\label{keybound} 
 Let $L_{a,b}, a\ge 1, b>ae$ denote a pencil of curves in complete linear system $aE+bF$ on polarized surface $(F_e,\CO_{F_e}(E+kF), k>e,e\ge 0$ 
  i.e. a line in 
$\PP(H^0(F_e, \CO_{F_e}(aE+bF)))$.  Then for 
\begin{equation}\label{excessbound} n>max(2(k+1)-e,7)
\end{equation}
and all $a,b$ with either $a>1$ or $b>k$, there are no pencils  $L_{a,b}$ admitting $n$ members relative to which $L_{a,b}$ is essential 
and transversal (cf. Def. \ref{threshold}) and having classes of 
all irreducible components of these members in $\nabla(D_k)$.

Moreover, the same property of the pencils $L_{a,b}$ takes place in the range $n>6$ if either $a\ge 3$, or 
\begin{equation}\label{exception}
(a) \ a=1 \ \  and \ \  b > {{7k-e} \over 5}  \ \ or  \ \ (b) \ \ a=2 \  \  and  \ \  b>{{14k-e-8} \over 12}
\end{equation}
 where $k_F$ is defined just before 
the function $K(a,b,e,k_F)$ is introduced in the proof below. 
For pencils $L(2,b)$ with $b\le 7k_F-e+4$ (resp. $L(1,b)$ with $b>{{8k-e}\over 6}$)  
this property takes place for $n >7$. Hence the only pencils $L(a,b), a \ge 1, b>ae$, having $n$ members with irreducible components in $\nabla(D_k)$
with $n$ satisfying inequality (\ref{excessbound}) are the pencils $L(1,b), b\le k$.
\end{thm} 

\begin{rem} The above theorem makes explicit  the {\it trichotomy for pencils on scrolls} (cf. \cite{cogome}, \cite{handbook}). For $M_2(\nabla(D_k))=max({2(k+1)-e,7})$ the only 
pencils with $n>M_2(\nabla(D_k))$ reducible members having classes of irreducible components in $\nabla(D_k)$ are the 
pencils in linear systems from $\nabla(D_k)$. There are at most finitely many pencils with more than $M_1(\nabla(D_k))=7$ 
members with 
such property 
(notations for the constants are from Theorem \ref{paperwithjose}).The restrictions on the pencils assuring their finiteness given by (\ref{exception}). 
The inequalities in Theorem \ref{keybound} are not sharp, i.e. many $L_{a,b}$ with $n$ reducible members 
below the bounds described in the Theorem, do not exist, 
 but verifying this becomes more and more the study of special cases. 
 
 There may be infinitely many pencils
with $n \ge 3$. 
For example, the pull back of the pencils $a(x^d-y^d)+b(y^d-z^d)=0, d>1$ from a plane 
in $\PP^{2k-e+1}$ to the scroll $S_{k-e,k} \subset \PP^{2k-e-1}$
produces a pencil with 3 members with all irreducible components in $\CO(E+kF)$.  It would be interesting to obtain precise information 
about pencils $L(a,b)$  with irreducible components in $\nabla(D_k)$ for $n \le 7$ or $n>7$ but $(a,b)$ are outside of the range \ref{exception}, when the pencils are not composed of
the elements in a linear system with class in $\nabla (D_k)$.
\end{rem}

\begin{proof} The argument follows the idea used in \cite{cogome} but, dealing with concrete family of surfaces, we can be more specific. 
We use the bound  telling that the Euler characteristic of the blow of of $F_e$ at the base points of $L_{a,b}$ is at least the Euler characteristic of 
the sum of the Euler characteristics of the completely decomposable fibers and the data of fibration with degenerate fibers 
over the complement in the parameter space of $L_{a,b}$ to the set of reducible fibers. 

Consider the blow up $\tilde F_e$ of $F_e$ at the base points of $L_{a,b}$ and the corresponding morphism
$\pi_{a,b}: \tilde F_e\rightarrow \PP^1$.  The number of base points is $2ab-a^2e$. 
Let $x$ be the number of fibers of $\pi_{a,b}$ with al irreducible components having classes in  $\nabla(D_k)$. 
Let $C_{gen}$ be generic fiber of $\pi_{a,b}$ 
The bound is obtained from inequality (\ref{keyinequality}) below (cf. \cite{barth}, Ch. III, Prop. 11.4), where $e(C_{gen})$ and $e(C_{s})$ are the Euler characteristics of {\it non-intersecting} (due to transversality condition on the pencil) generic and a singular members respectively,  the summation 
is over the set $S$ of singular members with all their irreducible components in $\nabla (D_{k})$ and $ {\rm Card} S=x$: 
\begin{equation}\label{keyinequality}
     (2-x)e(C_{gen})+\sum_{s \in S} e(C_{s}) \le e(\tilde F_e)=4+2ab-a^2e
\end{equation}
Recall that (\ref{keyinequality}) is immediate consequence of additivity of the Euler characteristics combined with semi-continuity.
Indeed, the difference between the right and left hand sides  is the sum of the Milnor number of singularities of singular members of the pencils
with irreducible components that do not have the classes of irreducible components in $\nabla(D_k)$.

Since by Bertini's theorem (cf. \cite{hart} Remark 10.9.2) $C_{gen}$ is smooth, the adjunction formula yields that 
\begin{equation}\label{degreegenfiber}
-e(C_{gen})=deg K_{C_{gen}}=2ab-a^2e+ae-2b-2a
\end{equation}
Let us estimate $e(C_s)$ in terms of $a,b,e$ and the data $k_{F_s}, k_{i,s},\delta_E$ which we now introduce. The description of the set $\nabla(D_k)$ implies that irreducible components of $C_s$ are 
all rational curves and $C_s$ may have $k_{F_s}$ components that are the fibers of the scroll, $\delta_E=0,1$ components 
that is the directrix $E$ of the scroll, and the rest are irreducible components having classes in $E+k_{i,s}F, e \le k_i <k$. This data is subject to constraint:
\begin{equation}\label{constraint}
k_{F_s}+\sum_{i=1}^{a-\delta_E} k_{i,s}=b \ \ \ k \ge k_{i,s} \ge e \ \
\end{equation} 
The Euler characteristics of the curve $C_{s}$ depends only on the Euler characteristics of the components and 
the number of intersections of irreducible components of $C_s$. 
For a fixed set of components, $e(C_s)$ for non-transversal intersections is not smaller than for a curve in which the 
rational components of $C_s$ are smooth and intersect transversally. Hence (3) implies the same inequality when $C_s$ denotes 
the replacement of singular fiber by the curve with smooth components and transversal intersections which we will assume from now on. 

In the latter case the number of intersection points of the components of $C_s$
is given by sum  $\sum_{i<j}C_iC_j$ where irreducible components $C_i$ of the singular fiber $C_s$ assumed to be ordered according a fixed  
order of their classes in the sum $$C_s=\delta_EE+k_{F,s}F+\sum_{i=1}^{a-\delta_E}(E+k_{i,s}F)$$ 
This sum of intersections of the classes of components is given by (recall that $E^2=-e, EF=1, (E+k_iF,E+k_jF)=k_i+k_j-e$):
$$I_s=k_{F,s}a+\delta_E(-e(a-\delta_E)+b-k_{F,s})-{1\over 2}(a-\delta_E)(a-\delta_E-1)e+\sum^{a-\delta_E}_{i,j, i <j} (k_{i,s}+k_{j,s})
$$
Using the constraint (\ref{constraint}) one has $\sum^{a-\delta_E}_{i<j}(k_{i,s}+k_{j,s})=(a-\delta_E-1)(b-k_{F,s})$ and the upper bound 
the number of intersections of the components is 
\begin{equation}\label{numberintersections}
I_s=k_{F,s}a+\delta_E(-e(a-\delta_E)+b-k_{F,s})-{1\over 2}(a-\delta_E)(a-\delta_E-1)e+(a-\delta_E-1)(b-k_{F,s})
\end{equation}
The lower bound for the Euler characteristic of the singular fiber is 
\begin{equation}\label{lowerbound}
2(a+k_{F_s})-I_s
\end{equation}
Substituting (\ref{lowerbound}) using value of $I_s$ given by (\ref{numberintersections}), (\ref{degreegenfiber}) into (\ref{keyinequality})
and replacing each $k_{F,s}$ by $k_F=min_s k_{F,s}$, we have:
\begin{equation}
x(2ab-a^2e+ae-2b-2a+2(a+k_{F_s})-[k_{F_s}a+\delta_E(-e(a-\delta_{E})+b-k_{F_s})-{1\over 2}(a-\delta_E)(a-\delta_E-1)e+
(a-\delta_E-1)(b-k_{F_s}))
\end{equation}
$$(a-\delta_E-1)(b-k_F))]
\le 4+2ab-a^2e+2(2ab-a^2e+ae-2b-2a)
$$
After simplification (using $\delta_E^2=\delta_E$)
\begin{equation}x(ab-b+k_F-{e\over 2}a(a-1)) \le 6ab-3a^2e+2ae-4b-4a+4
\end{equation}
Hence one has the bound $x<K(a,b,e,k_F)$ where $k_F=min_s k_{F_s}$ and 
\begin{equation}\label{functionK}
   K(a,b,e,k_F)= {{6ab-3a^2e+2ae-4b-4a+4} \over {(ab-b+k_F-{e\over 2}a(a-1))}} \ \ a>1 \ or \  b>k \ and \  k_F<k 
\end{equation}
(the last inequalities are consequences of the assumption that $L_{a,b}$ is not in a linear system in $\nabla(D_k)$). 
For fixed, $a,e,k_F$, as function of $b$, $K(a,b,e,k_F)$ is increasing with supremum of its valued being $8$. 
Hence the largest integer value of $x$ is $7$. Moreover, for $a \ge 3$, $K(a,b,e,k_F) <7$ i.e. the largest integer value of $x$ is 6. 

In the remaining case $a=1$, observe that summation in (\ref{constraint}) has only one term,   
that a curve $C_s \in \vert E+bF \vert$ having components in $\nabla(D_k)$ has at least $b-k$ irreducible components with numerical class $F$
and therefore it follows from (\ref{functionK}) that 
\begin{equation} 
x \le {{2b-e} \over {k_F}} \le {{2b-e} \over {b-k}} \le 2(k+1)-e 
\end{equation} 
with equality for $b=k+1$. Moreover, for  
$b\ge {7k-e \over 5}$ one has $x \le 6$.  
Similarly, for $a=2$ the summation index $i$ in (\ref{constraint}) takes at most two values,  
$k_F \ge b-2k$ and $x \le {{2b-8e-8} \over {b-e+k_F}} \le {{2b-8e-8} \over {2b-e-2k}} <7$ if and only if $b>{{14k-e-8} \over 12}$.
 The claims of the theorem follow. 
\end{proof}

\section{Pencils of curves with reducible fibers and arrangements of hyperplanes}

For a smooth quadric $W$ in $\PP^3$, the collection $\nabla$ of classes consisting of $\CO_{W}(1)$ and the classes of fibers of the rulings is saturated.
For the linear system $\CO_W(2)$ the bound in Theorem \ref{keybound} gives $6$ as the maximal number of elements in the pencils in this linear system 
with components in $\nabla$ (indeed, in the notation of the Theorem \ref{keybound}, we consider the case $e=0$ and so for 
the pencil $L_{2,2}$ a possibility of having 6 reducible members is not excluded.  
Here we give an explicit construction of a pencil on a smooth quadric in $\PP^3$ and 
a linear system of curves cut by quadrics with 6 members 
that are unions of plane sections i.e rulings of $W$.

Consider the net of diagonal quadrics (cf. \cite{dolg}): 
\begin{equation}
  Q_1: \sum_0^3 x_i^2=0, \ \ Q_2: \sum_0^3 a_ix_i^2=0, \ \ Q_3: \sum_0^3 b_ix_i^2=0
\end{equation} 
The discriminant of the net of quadrics spanned by $Q_i$ is the union of four lines (quadrilateral). 
Consider the arrangement of twelve planes in $\PP^3$ formed by irreducible components of 6 quadrics $V_i, i=1,\cdots 6$ of rank 2, corresponding to 
the vertices of this quadrilateral. 
Let $W$ be a smooth quadric in the net and $L$ be a generic pencil in this net. Let $l_i$ be the pencil of quadrics in $\PP^3$ containing $W$ and $V_i, i=1,\cdots 6$.
$V_i\cap W$ is the base locus of $l_i$ consisting of a union of two hyperplane sections of $W$.
Notice that the zero set of the members  $L\cap l_i, i=1,\cdots 6$ of $L$ contain the base locus of the corresponding $l_i$  i.e.  is a union of two hyperplane sections of $W$. In particular, the pencil by induced by $L$ on  
$W$ has 6 reducible members $L\cap l_i$. 
Hence we obtain: 

\begin{thm} For a saturated collection $\nabla$ on a smooth quadric $W$ in $\PP^3$, consisting of hyperplane section and the lines there exists 
a pencil in $H^0(W,\CO_W(2))$ with 6 reducible members with components in $\nabla$. 
\end{thm} 

\begin{rem} Smooth members of the net above fail to be transversal to the strata of the arrangement of 12 planes. 
Though each line of multiplicity 3 in this arrangement contains 2 base points and hence 
is transversal to smooth quadric, and  
generic smooth member intersects transversally 2-dimensional strata of the arrangement, transversality fails for 0-dimensional ones given by the 
base points of the net. In particular, Lefschetz hyperplane section theorem does not provide a relation between the fundamental group of the complement to  the arrangement of planes in $\PP^3$ and the complement in $W$ of the union of the reducible members of the  pencil induced on $W$. 
The fundamental group of the latter complement is a semidirect product of free groups $F_5$ and the fundamental group of the  
 elliptic curve with 8 points removed i.e. $F_9$, as follows from exact homotopy sequence of a fibration.   
\end{rem} 

\begin{rem} For complete classification of Lefschetz pencils of type $(2,2)$ cf. \cite{hamada}.
\end{rem} 

\section{Linear systems with small degree of discriminant and stratification of dual variety of a scroll}\label{smalldegree}

\subsection{Polarized simply connected surfaces with degree of discriminant at most 12.}
\begin{prop} 
Let $L$ be a very ample line bundle on a smooth surface $X$, let $\Delta=L^2+2-dim H^0(L)$ be the Fujita's $\Delta$-genus of $(X,L)$ and  $g(X,L)$ be the arithmetic genus of 
an element in $L$. Let $D$ denote
the degree of the discriminant hypersurface in $\PP(H^0(L))$. Then 
\begin{equation}\label{degdisc1}
D=c_2+4(g(X,L)-1)+L^2=c_2+4g(X,L)+\Delta+{\rm dim} H^0(X,L)-6
\end{equation}
\end{prop}

\begin{proof} It follows from \cite{DL} that 
\begin{equation}\label{degdiscdolg}
D=c_2(X)+2K_XL+3L^2
\end{equation}
 Adjunction formula and definition of $\Delta$ give (\ref{degdisc1}).
\end{proof} 

\begin{cor}\label{smalldiscrimlist}  (a) If a polarized pair $(X,L)$ has the degree of discriminant not exceeding 12 then $\Delta=0$ or $1$.
If $X$ is simply connected with $D \le 12$ and $\Delta(X)=1$ then $X,L$ is a del Pezzo surface and the degree $D$ of discriminant is 12. 

(b) Complex polarized surfaces having $\Delta$-genus zero are, 
\begin{itemize} 
\item (i) $\PP^2, \CO(1)$, 
\item (ii) $\PP^2,\CO(2)$,
\item  (iii)
$V_2,\CO(1)
$. 
\item (iv) $(F_e,\CO(E+kF))$ where $F_e$ is a Hirzebruch surface, $E$ is its section with negative selfintersection, 
$F$ is the class of a fiber and $k>e$.  
\end{itemize} 
The degrees of discriminants are as follows (i) discriminant is empty, in the case  (ii) the degree is 3 and in the case (iii) it is 2. The surfaces in (iv) 
have both, the degree $L^2$ and the degree of the discriminant  equal to $D=2k-e$. 
\end{cor}
\begin{proof} To show (a), notice that  if $\Delta \ge 2$ then also $g \ge 2$ (cf. \cite{fujita0}, \cite{fujita1}, \cite{lanteri}) Sec.1.1).
A simply connected surface with $c_2=3$ i.e. ${\rm rk} H_2(X,\ZZ)=1$ must be a homotopy projective plane and hence is biholomorphic to $\PP^2$ (cf. \cite{yau}).
Therefore we can assume $c_2\ge 4$. Since surfaces in $\PP^3$ with $\Delta \ge 2$ have the degree of discriminant greater than 12, we can assume that 
$dim H^0(X,L)\ge 5$.  Since for such simply connected surface $c_2 \ge 4$ and $dim H^0(X,L)\ge 4$  the right hand side in (\ref{degdisc1}) is at least 13. 
If $\Delta=1$ then either we have a del Pezzo surface (i.e. the polarization is by the anti-canonical bundle) and the degree of discriminant is 12 or one has  $g\ge 2$ and $L^2 \ge 6, c_2\ge 4$ (cf. \cite{fujita1}, Th.1) in which case the first equality (\ref{degdisc1}) shows that the discriminant is bigger than 12. 

(b) is the summary of Fujita's classification (or much earlier classical works) with the rest following from \cite{fujita0} Thm 2.1, 2.2, 3.8 and the discussion of ruled surfaces in \cite{hart} Ch 5. 
The last claim, i.e. the equality of the degree of a scroll and the degree of its dual variety was known classically \cite{salmon} and more recently observed in \cite{ragni} 
\footnote{G.Salmon \cite{salmon} vol.1 p.121 attributes to Cayley that the degree of a ruled surface in $\PP^3$ coincides with the degree of its dual. Salmon-Cayley's surfaces are generic projections of smooth scrolls in high-dimensional projective spaces. Since the dual of generic projection can be identified with 
generic linear section of dual of the smooth surface one sees the equivalence of Cayley statement to the one in the Corollary \ref{smalldiscrimlist}. 
The latter relation was well known classically for example in the works by  O. Zariski. For a modern exposition cf. \cite{yilong}.}
\end{proof}

\begin{cor}\label{atmost6} Complex polarized surfaces with degree of discriminant at most 5 (resp. 11) are the surfaces (i),(ii),(iii) and scrolls 
$S_{a,b}, a+b \le 5$ (resp. $S_{a,b}, a+b \le 11$)
\end{cor}
\begin{proof} It follows from (\ref{degdiscdolg}) that degree of the discriminant of the scroll $(F_e,\CO(E+kF))=S_{k-e,k}$ is $2k-e$. 
This  immediately implies the claim of the corollary. 
\end{proof} 

\begin{rem} There are 30 scrolls with degree of discriminant at most 11 and the Corollary \ref{atmost6} gives a complete list, correspond to 
the solutions $a+b \le 11, a \le b$.
\end{rem}

\begin{rem} Fujita also considers {\it prepolarized} manifolds (cf. \cite{fujita0} Prop. 1.4) $(X,L)$
where $(X,L)$ without assuming $L$ being ample. 
Some have small degree of discriminant  (e.g.  $(F_1,\CO (E+F))$) and are not considered in this paper. 
\end{rem}

\subsection{Stratification of dual varieties of scrolls.} 

\begin{thm}\label{stratification} Dual variety of a scroll $S_{k-e,k} \subset \PP^{2k-e+1}$ admits a 
stratification with smooth strata  parametrizing hyperplanes intersecting $S_{k-e,k}$ 
along the the curves of the following types: 

(a) Strata $S_{\pi_s(l)}$, where $\pi_s(l)$ is a partition of $l, 1 \le l \le k-e$  into $s$ parts,  
with points corresponding to hyperplane sections that are unions of smooth curves in the linear system $E+(k-l)F$ and $s$ distinct fibers of $F_e$ 
with multiplicities $m_i, m_{i_1}+\cdots+m_{i_s}=l$) (i.e. $s$ is the length of partition of $l$). 

(b) Strata $T_{\pi_s(k)}$, where as in (a) $\pi_s(k)$ is a partition of $k, \sum_{i=1}^s m_i=k$ having length $s$, with points corresponding to hyperplane sections that are reducible curves which are a union of the exceptional curve of $F_e$ (the directrix) and $s$ distinct fibers (rulings) with multiplicities $m_i$.

These strata correspond to the strata of the discriminant of binary forms of degree $k$ in the case of the strata $T_{\pi_s(k)}$ and to the strata of 
the resultant hypersurface of polynomials of degree $k-e$ and $k$, in the case of $S_{\pi_s(l)}$.

2. The dimensions of these strata are given as follows
\begin{equation}
   dim S_{\pi_s(l)}= 2(k-l)-e+1+s, \ \ \ \  dim T_{\pi_s(k)}=s. 
\end{equation}
The adjacency is given by the refinement order of partitions (cf. \cite{stanley} Sect. 7.2)  i.e. a stratum corresponding to a partition $\lambda(k)$
is in the closure of a stratum $\mu(l), 1 \le l \le k-e$ iff $\mu(l)$, extended by $s(\lambda)-s(\mu)$ zeros and $\lambda(k) \ge \mu(l)$.

3. The multiplicity of a smooth point of a stratum  $S_{\pi_s(l)}$ (respectively stratum $T_{\pi_s(k)}$) is 
equal the sum of multiplicities $l$ (respectively $k$) in the partition. 

\begin{rem}
Claim (2) of the Theorem implies that $S_{1^2}$ is the only codimension 1 stratum of the discriminant of $\vert E+kF \vert$ (having as the normal slice an ordinary plane curve singularity of multiplicity 2)  and that there are two codimension 2 
strata: $S_{1^3}$ (normal slice is the normal crossing surface singularity of multiplicity 3) and $S_{2^1,1^1}$ (normal slice is the pinch point). 

Moreover, it can be seen that the canonical desingularization of the dual variety $S_{k-e,k}^{\vee}$ given by the projection of the incidence correspondence 
$I \subset S_{k-e,k} \times (\PP^{2k-e-1})^{\vee}$ on the second factor \footnote{the fibers of projection onto the first factor are the projective spaces parametrizing the hyperplanes containing the tangent plane at a point. In particular $I$ is smooth}
\begin{equation}
\begin{matrix} && I & & \cr 
                                                                     & \swarrow & & \searrow&  \cr
                                                                     S_{k-e,k} & & & & S_{k-e,k}^{\vee} \subset (\PP^{2k-e+1})^{\vee} \cr
\end{matrix}                                                                     
\end{equation}
has the explicit structure of the fibers 
depending on the strata of the above stratification. More precisely, the fibers over the a stratum $S_{\pi_s(l)}$ or $T_{\pi_s(k)}$
have dimension zero if and only if all parts are $1$ and the dimension $1$ otherwise.

 This resolution is a projective version of the morphism studied in (\cite{joun} (1.1) for  more general resultants.
 Note also that an explicit calculation of the stratum of codimension 2 (pinch point) for affine resultant $R_{2,2}$ was done in  \cite{choudary}.
\end{rem}

\end{thm} 

\begin{proof} A hyperplane section of $S_{k-e,k}$ is singular if and only if it is a union of rulings and the residual curve that is transversal to the ruling. 
Indeed, if either a ruling is not a component of a hyperplane section or it is a component but the residual curve is not transversal to this rulings, then the intersection index of this ruling and the hyperplane section at the tangency point (which is singular point of the intersection) must be at least 2, while $(F,E+lF)=1$.
The first part of claim (a) follows since $H^0(F_e,\CO(E+(k-l)F))=2(k-l)-e+1$ for $k-l \ge e$ and hence each of $S_{\pi_s(l)}$ is non empty. 
Specifically, curves in this stratum can be constructed as follows. The dual of the directrix, that is a rational normal curve spanning a $\PP^{k-e}$, is a hypersurface in $(\PP^{k-e})^{\vee}$, can be identified with the discriminant of the spaces of binary homogeneous forms 
having degree $k-e$. A hyperplane, in $\PP^{k-e}$ intersecting the directrix at $s$ points having order of tangency $m_1,\cdots m_s, \sum m_i=l$, 
and a hyperplane in $\PP^k \subset \PP^{2k-e+1}$, containing rational normal curve of degree $k$ belonging to $S_{k-e,k}$, and having 
the same orders of tangencies at $s$ points corresponding to $s$ points on the directrix, span the linear subspace $\PP^{2k-e-1}$ of
 $\PP^{2k-e+1}$.Then the intersection of generic hyperplane in $\PP^{2k-e+1}$ containing this $\PP^{2k-e-1}$ with the scroll is a curve in the 
stratum $S_{\pi_s(l)}$ as described in the (a). In particular, assigning to a hyperplane corresponding to a point in $S_{\pi_s(l)}$ $s$ distinct points on 
the directrix gives a locally trivial fibration of this stratum over the (smooth) configuration space of $s$ points in $\PP^1$ with fiber that can be 
identified with the fibration over the space of hyperplanes in $\PP^k$ containing a subspace $\PP^{k-e-1}$ with fiber $\PP^1$ corresponding 
to hyperplanes in $\PP^{2k-e+1}$ containing $\PP^{2k-e-1}$ (spanned as described by hyperplanes in $\PP^{k-e}$ and $\PP^k$). 

The strata $T_{\pi_s(k)}$ consist of the points of dual variety 
corresponding to hyperplanes in $\PP^{2k-e+1}$ containing the span $\PP^{k-e}$ of the directrix of $S_{k-e,k}$  
and a hyperplane  $\PP^{k-1}$ in  $\PP^k \subset \PP^{2k-e+1}$ containing a rational normal curve of degree $k$ which belong to a  stratum of the dual
variety of the rational normal curve of degree $k$. As above, in the case of directrix, this dual hypersurface can identified with the discriminant hypersurface of the space of binary homogeneous forms of degree $k$ which in turn corresponds to a partition of $k$ \footnote{the argument here is close to the one in Prop. 1.5 in \cite{ciliberto} but the 
question addressed there is slightly different: for which $m$ any $m$ generic points on are the tangency points of a hyperplane}.
and the assertions in (b) follow. 

 For  convenience recall the relation between the dual variety of a scroll and 
 the resultant of binary forms of degree $k-e$ and $k$ (see \cite{kapranov} for a much general discussion of resultants). Consider scroll in $\PP^{2k-e+1}$ given by equations:
\begin{equation}\label{scrolleq}
 x_0=su^{k-e}, \cdots x_i=su^{k-e-i}v^i, \cdots x_{k-e}=sv^{k-e}, 
 \end{equation}
$$x_{k-e+1}=tu^k, \cdots x_{k-e+1+j}=tu^{k-j}v^j, \cdots x_{2k-e+1}=tv^k$$
Checking in local charts the smoothness of the intersection of a hyperplane $H_{{\bf a},{\bf b}}$  
given by $\sum_{i=0}^{k-e} a_ix_i+\sum_{j=0}^kb_jx_{k-e+1+j}=0$ with the scroll (\ref{scrolleq}) 
having the form 
$F(s,t,u,v)=sF_1(u,v)+tF_2(u,v)$ where
\begin{equation} 
F_1(u,v)=\sum_{i=0}^{k-e} a_iu^{k-e-i}v^i,  \ \ \ F_2(u,v)=\sum_{j=0}^kb_ju^{k-e+1-j}v^j=0
\end{equation}
one sees that for $t=u=1$ this intersection is singular iff $F_1(u,1)=F_2(u,1)=0$ and $s(F_1)'_u+(F_2)'_u=s(F_1)_v'+(F_2)'_v$. In particular  $(\cdots a_i,\cdots b_j\cdots)$ correspond to a point in the dual of the scroll iff it is in the resultant of $F_1(u),F_2(u)$. 
It is immediate that the line $(s,t)$-line $L_{u_0}$ corresponding to the common root $u_0$ of $F_1,F_2$ is an irreducible component of the intersection and 
that residual curve is transversal to this line. Moreover, the second differential is $sF_1+F_2$ at the intersection point of residual curve and $L_0$ has rank 1 iff $L_u$ has multiplicity 2. Continuing with the same calculations applied to ${1\over {u-u_0}}(sF_1(u)+F_u)$ i.e. arguing on the scroll $S_{k-e-1,k-1}$ one 
obtains an alternative argument for (a),(b). 

The formula for the dimension of the strata follows from the identity $dim \PP(H^0(F_e,\CO(E+(k-l)F)))=2(k-l)-e+1$ for $l \le k-e$ (cf. (\cite{coskun} Lemma 2.5) and the observation above showing that the data consisting of an element of this linear system,  the partition $\pi(l)$  
having length $s(\pi)$, and $s(\pi)$ intersection points of the multiple fibers with the directrix of the scroll specify uniquely the element of the stratum $S_{\pi_s(l)}$.  The argument for the strata $T_{\pi_s(k)}$ is the same.

The adjacency claim follows clear for pairs of strata of types $S,S$ and $S,T$. For pairs of  strata of type $S$ and $T$ since a union of an irreducible curve in linear system $E+(k-l)F$ and fibers of $F_e$ according to $\lambda(l)$, after generic degeneration of 
the former will be a curve in $T_{\mu_s(k)}$ corresponding to partition which the extension of $\lambda(l)$ given by $\lambda+1^{k-l}$. The adjacency of 
strata $T$ is given by classic order of partitions and the claim 2 follow. 

The claim 3 follows from \cite{aluffi} Proposition 2.1 expressing the multiplicity of a point on the discriminant of a very ample 
linear system on a surface $X$ corresponding to a singular, possibly non reduced, curve $C$ in this linear system
(and obtained, similarly to the calculations in section 3, as the degree of discriminant minus the number of singular members
in a generic pencil containing $C$ different from $C$). Let $C=\sum_i^s m_iC_i$ be a curve with irreducible components $C_i$ having respective multiplicities $m_i$, $C_{red}$ be the curve having $C_i$ as its irreducible components with multiplicity one and $\mu$ be the sum of Milnor numbers of all singularities of $C_{red}$. For a curve $G$, its class in $NS(X)$ we denote as $[G]$. Then  (cf. \cite{aluffi})
$$m_C=([C]-[C_{red}])(K_X+2[C]+[C_{red}])+\mu
$$ 
For a curve $C_{\pi_s(l)}$  in $E+kF$ corresponding to a smooth member of $S_{\pi_s(l)}$  $E+kF$, the class, one has 
$[C_{\pi_s(l),red}]=E+(k-l+s)F$ and $\mu=s$. In the case of the curve in the $T_{\pi_s(l)}$, one has the same class for the 
reduced curve as for a curve in $S_{\pi_s(l)}$ but $\sum \mu_i=s+k-l$. This shows the claim 3.

\end{proof}

\section{Extremal fundamental groups for surfaces with $\Delta=0$:  Quadrics and Veronese surface of degree 4.}. 

This section contains a calculation of presentations for the extremal groups (cf. Definition \ref{threshold}) for surfaces (i),(ii),(iii) in the list of the Corollary 
\ref{smalldiscrimlist}. The case (iv) is considered in the following section. 

\subsection{Extremal groups of $(V_2,\CO_{V_2}(1))$.}

It follows from Theorem \ref{keybound} that for saturated set $\nabla=(H,L_1,L_2)$ consisting of hyperplane sections and two rulings of a quadric ,  
any pencil containing more than 7 members, each being a union of divisors from $\nabla$, is a pencil of hyperplane sections. 
The dual variety of a smooth quadric $V_2\subset \PP^3$ is a smooth quadric. In particular there is only one equisingular type of pencils 
with reduced members having two singular members i.e. the line in dual space transversal to the dual quadric.  In particular,  for each $r > 6$, we have three $r$-extremal groups $P^i_r, i=0,1,2$
(cf. Definition \ref{threshold})  that are the fundamental groups of the complements to $r+1$ members of 
the generic pencil of hyperplane sections having $i$ singular hyperplane sections 
among $r+1$ deleted members.  

\begin{prop}\label{quadricextremal} One has the following isomorphisms: 
 \begin{equation} 
 P^2_r=\ZZ\times F_{r} \ \ P^1_r=F_{r}    \ \ P^0_r=F_{r}
 \end{equation}
\end{prop} 
\begin{proof} The base locus of a generic pencil of hyperplane sections of a quadric $V_2$ consists of 2 base points. Let $\tilde V_2$ denotes the blow up at these  points 
and $\tilde V_2\rightarrow \PP^1$ the map induced by the pencil. Each of 4 lines on $V_2$ forming a pair of singular members of the pencil contains 
exactly one base point of the pencil and hence on $\tilde V_2$ the proper preimage of each of these lines will have self-intersection index $-1$ while proper preimage of singular member remains to be a pair of intersecting rational curves. Proper transform of a smooth member of the pencil will have self-intersection $0$. 
Contracting two proper transforms of irreducible components singular members, one from each member, such that they intersect distinct exceptional curves (there are two choice for such contraction), we obtain again $\PP^1\times \PP^1$ (with generators linearly equivalent to images of exceptional curves 
and the proper transforms on $\tilde V_2$ of  smooth members of the pencil). The remaining irreducible components of proper transforms of singular members
after contraction become linearly equivalent to images of smooth members. Therefore the complement to the union of such $n$ members of generic pencil 
is biregular to $\PP^1\times \PP^1 \setminus (a \times \PP^1\cup b\times \PP^1 \bigcup_1^{r+1} (\PP^1\times p_i)$ which implies the first isomorphism. 

In the case when among missing $r+1$ members of the pencil of there is only one singular member then the complement can be viewed as the complement 
of the same pencil with missing $r+1$ members among which there are two singular members and one is filled in. The group of the complement before 
filling in was just  calculated. Filling in a singular fiber changes the  fundamental group by adding relations $m_i=1, i=1,2$ (cf. (\ref{killingmeridians}) and
Appendix, Theorem  \ref{ZvK})
where $m_i$ are the meridians of two
components of the singular member that is added. In birational transformation described above, the meridian of contracted component is isotopic to 
the parallel of remaining component \footnote{Locally, for the complement to the union of transversally intersecting smooth curves in their 
regular neighborhood, in the fundamental group of the complement, that is $\ZZ^2$ the standard generators are 
meridians of components but also can be viewed as push outs of circles in each of components i.e. the parallels.} that become the meridian of the  image of exceptional curve which was intersected by the contracted component. Hence, in this case, the complement has the fundamental group of a product of $\CC$
and $\PP^1$ minus $r+1$ points which shows the second isomorphism. The third one follows immediately by the same argument (but identification is 
with the fundamental group of product of $\PP^1$ and complement in  $\PP^1$ to $r+1$ points.  
\end{proof}

\subsection{Number of extremal groups on Veronese surface $(\PP^2,\CO(2))$}. 

In this section we consider  Veronese surface of degree 4 with saturated system 
$\nabla=([1,],[2]) \in Pic(\PP^2)$ i.e. the conic-line arrangements. 
We calculate the number of $r$-extremal groups for $r>6$ i.e. the 
fundamental groups of complements to arrangements of lines and quadrics in $\PP^2$, satisfying conditions (a),(b) in Prop. \ref{pencils}, admitting a free group of rank greater than $6$ as its quotient. This case of Veronese surface of degree 4 with saturated system 
$\nabla=([1],[2]) \in Pic(\PP^2)$ is the first after  $(\PP^2,[1])$ and is continuation of discussion started in section 4.2 in \cite{cogome} (where the threshold $6$ was obtained) and Sect. 5 in \cite{handbook}. Recall that in the of arrangement of lines the only extremal groups are the free groups (cf. \cite{libyuz}).
\begin{prop} 
There are 4 families of extremal groups $Q^i_{r}, i=0,1,2,3, r+1 \ge 6$  that are the fundamental groups of the complement to a 
union of $r+1>6$ quadrics in a pencil, among which $i$ quadrics are 
singular.  
\end{prop} 
\begin{proof} It follows from the results of section 4.2 in \cite{cogome} or Theorem \ref{keybound} and Prop. \ref{pencils} that a curve $D$ having
an extremal fundamental group of the complement is a union of $r+1 > 6$ smooth or reduced reducible quadrics. Since the degree of the only stratum of the discriminant of 
linear system $\PP(H^0(\PP^2,\CO(2))$ with only reducible and  reduced components 
is $3$, it follows that the only pencil of quadrics having no multiple fibers is the generic one intersecting this stratum at smooth points transversally, and a possible number of singular members of the pencil is $i, 0 \le i \le 3$ which shows the claim. Explicit forms of all possible pencils are given in \cite{dolgclass} 
Lemma 9.3.8. 
\end{proof}																															
\subsection{Presentation of the groups $Q^i_r$}

\begin{thm}\label{presentationquadrics} Let $Q^3_r, r \ge 2$ be the fundamental group of the complement to a reducible plane curve that is the union of 3 singular members  and $r-2$ smooth members  
in  a generic pencil of plane quadrics \footnote{Since $4,5$ and $6$ are below the threshold of the
saturated subset $\{[1],[2]\} \subset Pic(\PP^2)$, $Q^3_r, 2 \le r \le 5$, the groups $Q^3_r$ are not extremal in the sense of Definition \ref{threshold}.}

$(a)$ The group 
$Q^3_{3}$  admits a presentation with (defined below geometrically) generators: 
$$a_1, a_2,a_3, x_1,x_2,x_3$$
and 9 relations, 3 of which are the commutativity relations and 6 are obtained by equating cyclic permutations of factors:
\begin{equation}\label{presentationQ34}
[x_3a_3x_3^{-1},x_2]=[a_2.x_1]=[a_1x_3]=1
\end{equation}
$$(a_1x_1x_2) \ \ \ (a_2x_2x_3) \ \  (a_3x_1x_3)
$$

$(b)$ For $r \ge 3$, a group $Q^3_{r}$ admits a presentation:
\begin{equation}\label{presentationQ3} 
\{a_i,x_j,i=1,\cdots r, j=1,2,3 \vert [x_3\alpha_3x_3^{-1},x_2] =[a_2,x_1]=[a_1x_3]=[a_i,x_j]=1 \ 4\le i\le r
\end{equation}
$$ (a_1x_1x_2) \ \ \ (a_2x_2x_3) \ \  (a_3x_1x_3) \}
$$
Alternatively, $Q_r^3, r \ge 3$ is the quotient of amalgamated product $Q_3^3*F_{r-3}$ by normal subgroup generated by 
commutators $[x_j,a_i], j=1,2,3$ and $a_i, i=4,\cdots r$ are generators of free group $F_{r-3}$.   

$(c)$ The remaining groups $Q_r^i, 0\le i<3$ are the following quotients of $Q_{r+1}^3$: 
\begin{equation}\label{presentationQ2}Q^2_{r+1}=\ZZ^2 \oplus F_{r} \ \  Q^1_{r+1}=\ZZ \oplus F_{r}, \ \ Q^0_{r+1}=\ZZ_2\oplus F_{r}
\end{equation}
\end{thm}
\begin{proof} Group $Q^3_3$ is the fundamental group of the complement in $\PP^2$ to a union $D^3_3$ of 3 singular members of a pencil of 
quadrics without multiple fibers and a smooth quadric in the same pencil. Equivalently, the conic-line arrangement $D^3_3$ is a union 
 Ceva arrangement of 6 lines having 4 generic points as its points of multiplicity 3 and a smooth quadric in the pencil having these 4 points as the base points. Blowing up $X=\PP^2$ at these 4 points and denoting the exceptional curves $E_i, i=1,\cdots ,4$ on resulting surface $\tilde X$, we consider regular map $\pi: \tilde X \setminus (\pi^{-1}(Crit \cup P) \bigcup_1^4 E_i)) \rightarrow \PP^1$, where $Crit \subset \PP^1$ is the set of critical values of $\pi: \tilde X\rightarrow X$, $P \in \PP^1\setminus Crit $. We apply the Zariski-van Kampen method 
 to obtain the presentation (a) as described in Appendix, Theorem \ref{ZvK},(\ref{redsection222}), (in the case when $G$ is empty). 
 
 Let us designate one of the curves $E_i$ as the section at infinity $E$ in (\ref{redsection222}),Theorem \ref{ZvK} and assume that $C^{\infty}$-section $\tilde E$ of the normal bundle of $E$, intersects $E$ transversally at $E\cap \pi^{-1}(\bar B')$ ($\bar B' \in \PP^1\setminus Crit$, cf. notations of Theorem \ref{ZvK}).  Consider geometric monodromy (cf.(\ref{geommonodromyformula})): $$\beta: \pi_1(\PP^1\setminus (Crit \cup \bar B'), d)\rightarrow Mod (S^2,[4], \partial(T(E)) \cap \pi^{-1}(d_0), \bigcup_1^3 E_i \cap \pi^{-1}(d))$$ where $d_0 \in \PP^1\setminus Crit \cup \bar B'$ is the base point.  The target of 
 $\beta$ is the mapping class group of a surface with one boundary component and 3 punctures in interior points, that is 2-sphere with removed open disk: $\pi^{-1}(d) \setminus T(E) \cap \pi^{-1}(d)$ centered at intersection point $E\cap \pi^{-1}(d_0)$ with marked points being the intersections of $\pi^{-1}(d_0)$  with remaining 
 3 exceptional curves $E_i, i=1,2,3$.

 \begin{figure}
\centering












     
    
     
\begin{tikzpicture}[
  dot/.style={circle,draw,fill=white,minimum size=1.0 mm,inner sep=0pt}
]

\node[dot] (a1) at (0,0) {};
\node[dot] (a2) at (.7,0) {};
\node[dot] (a3) at (1.4,0) {};

\draw[thick,red] ($(a1)!0.5!(a2)$)
  ellipse[x radius=.53,y radius=.18];

 \node at (0.35,-0.5) (b) {$s_1^2$};

\node[dot] (b1) at (2.8,0) {};
\node[dot] (b2) at (3.5,0) {};
\node[dot] (b3) at (4.2,0) {};
   
 \draw[thick, red, smooth cycle] plot coordinates {
    (2.7, 0.1)   
    (3.5, 0.3)   
    (4.3, 0.1)   
    (4.4, -0.05)   
    (4.2, -0.2)
     (3.85,-0.1)
     (3.5,0.13)
    (2.8, -0.2)   
};

      \node at (3.5,-0.5) (b) {$s_1^{-1}s_2^2s_1$};
    

\node[dot] (c1) at (5.6,0) {};
\node[dot] (c2) at (6.3,0) {};
\node[dot] (c3) at (7.0,0) {};

\draw[thick,red] ($(c2)!0.5!(c3)$)
  ellipse[x radius=.53,y radius=.18];

  \node at (6.6,-0.5) (b) {$s_2^2$};

\end{tikzpicture}
\caption{Vanishing cycles}
\label{fig:vancycle} 
\end{figure}
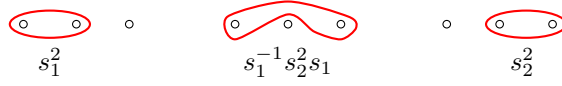

  The latter mapping class group can be identified with the pure braid group $P_3$ (cf. \cite{FM}, Sect. 9.1.4). The monodromy of a generic pencil of quadrics appeared in the literature on several occasions (cf. e.g. \cite{auroux}, Sect. 3.1) and is given by 3 Dehn twists around cycles surrounding 3 pairs among 
 mentioned 3 point. These 3 Dehn twists lead to lantern relations and their presentation using arrangements of 3 lines in $\CC^2$ is contained in  \cite{eko}.
 It follows from \cite{eko} that in standard geometric description of Artin's braid group (cf. \cite{FM}), these vanishing cycles and the corresponding braids 
 are those on Figure \ref{fig:vancycle}.   
Algebraic description of these Dehn twist as the elements of $P_3$ in terms of standard generators of the Artin's braid group are (cf. \cite{FM})   is as follows (for appropriate choice of paths $a_i  \in \pi_1(\PP^1\setminus Crit \cup \bar B',d_0)$):
\begin{equation}\label{braidmonodromyquadrics}
 \beta(a_1)=s_1^2, \ \ \beta(a_2)=s_1^{-1}s_2^2s_1, \ \  \beta(a_3)=s_2^2 
\end{equation}		 
The product is the generator of the center $\Delta^2$ of $P_3$ (this identity is the lantern relation \cite{eko}) 
and the classical formulas for the action of 
 the action of standard generators $s_i, i=1,\cdots n-1$ of the braid group $B_n$ on generators $x_i, i=1,\cdots n$ of the free group i.e.: 
 \begin{equation}\label{artinaction}
  s_i(x_i)=x_ix_{i+1}x^{-1}_i, s_i(x_{i+1})=x_i, \ s_i(x_j)=x_j, \ j\ne i,i+1
  \end{equation}
  give the following presentation of the group $Q^3_3$.  It has 6 generators $x_1,x_2,x_3 \in \pi_1(\pi^{-1}(d_0)\setminus E_i\cap \pi^{-1}(d_0)), 
  a_1,a_2,a_3 \in \pi_1(\PP^1\setminus Crit \cup \bar B',d_0)$ and 9 braid monodromy relations:
\begin{equation}
r(i,j): \ \ \  \beta(a_i)x_j=a_i^{-1}x_ja_i
\end{equation}
or explicitly:  
$$
r(1,j): a_1^{-1}x_1a_1=s_1^2(x_1)=x_1x_2x_1x_2^{-1}x_1^{-1}; \ \ a_1^{-1}x_2a_1=s_1^2(x_2)=x_1x_2x_1^{-1}; \ \ \
a_1^{-1}x_3a_1=x_3
$$
Using the second relation to simplify the first, one obtains:
\begin{equation}\label{alpha1relations}
r(1,j):  a_1x_1x_2=x_2a_1x_1=x_1x_2a_1 \ \ \ [a_1,x_3]=1
\end{equation}
Similarly, 
$$
r(2,j): [a_2,x_1]=1; \ \ a_2^{-1}x_2a_2=s_2^2(x_2)=x_2x_3x_2x_3^{-1}x_2^{-1}; \ \  a_2^{-1}x_3a_2=s_2^2(x_3)=x_2x_3x_2^{-1}
$$
and hence
\begin{equation}\label{alpha2relations}																		
r(2,j):  a_2x_2x_3=x_3a_2x_2 =x_2x_3a_2; \ \  [a_3x_1x_3x_1^{-1}x_3^{-1},x_2]=1
\end{equation}
The last triple of relations corresponding to $a_3=s_2^{-1}s_1^2s_1$ is 
\begin{equation}\label{alpha3relations}
r(3,j): a_3x_1x_3=x_3a_3x_1=x_1x_3a_3 \ \ [a_3,x_2]=1
\end{equation}
This shows (a). 

Part (b) follows from Appendix, Theorem \ref{ZvK},(\ref{redsection222}), since the braid monodromy corresponding to 
smooth fibers of the pencil is trivial. The description as an amalgamated product is an immediate consequence of presentation (\ref{presentationQ3}). 


To obtain presentation for $Q^2_r$ we start with presentation for $Q^3_{r+1}$ and add relations corresponding to filling in one 
reducible fiber of the pencil applying (\ref{killingmeridians})).
In the case of generic pencil of quadrics, each singular fiber is an intersection of two lines and meridians of the components are the loops
$a, a'$ such that $a^{-1} a'$ is a vanishing cycles (cf. (\ref{killingmeridians})). Filling in the fiber of the pencil corresponding to generator $a_1$ and vanishing 
cycle $x_1x_2$ yields presentation for $Q^2_r$
\begin{equation} \{a_2,a_3,\cdots a_r, x_2,x_3 \vert a_2x_2x_3, a_3x_2^{-1}x_3, [x_2,x_3a_3x_3^{-1}]=[a_2,x_2^{-1}]=1, 
\end{equation}
$$ [a_i,x_2]=[a_i,x_3]=1, i \ge 4\}$$
which can be identified with $\ZZ^2\oplus F_r$.
Similarly, for  $Q^1_4$ one obtains:
\begin{equation}
\{a_1,a_2,a_3, a_4,\cdots, a_r, x_2,x_2,x_3 \vert a_1=a_2=1, x_1x_2=x_2x_3=1, a_i,x_2]=[a_i,x_3]=1, i\ge 4
 \}
\end{equation}
which shows that $Q^1_r=\ZZ\oplus F_{r}$ 
and similarly $Q^0_r=\ZZ_2\oplus F_r$.
\end{proof}

\begin{cor} The group $Q^3_2$ is the quotient of $Q^3_3$ by the relations in Appendix, Theorem \ref{ZvK} (\ref{redsection}) and is isomorphic to the quotient of pure braid group $P_4$ by its center.
\end{cor}

\begin{proof} It follows from Theorem (\ref{ZvK}) (\ref{redsection}) that filling in the smooth fiber of the pencil results in the relation $r_{\infty}: a_1a_2a_3=(x_1x_2x_3)^{-1}$ i.e.
$Q^3_2=Q^3_3/r_{\infty}$.  
Consider the quotient  $Q_3^2$ by the normal subgroup generated by $x_1,x_2,x_3$ which yields exact sequence:
\begin{equation}\label{sequenceQ^3}
   1 \rightarrow F_3 \rightarrow Q^3_2/r_{\infty} \rightarrow F_3/a_1a_2a_3 \rightarrow 1
\end{equation}
On the other hand the quotient of $P_4/Z(P_4)$ of pure braid group by its center fits the exact  sequence:
\begin{equation}\label{sequenceP_4}
   1 \rightarrow F_3 \rightarrow P_4/Z(P_4) \rightarrow P_3/Z(P_3) \rightarrow 1
\end{equation}
The right group in (\ref{sequenceP_4}) is isomorphic to $Ker(B_3/(\Delta^2) \rightarrow S_3)$ hence can be 
identified with the index 6 subgroup in $PSL_2(\ZZ)$ that is free on two generators.
The natural map between groups at the ends of the extensions (\ref{sequenceQ^3}) and (\ref{sequenceP_4}) yields the isomorphism between these extension.
since both have monodromy maps $F_2\rightarrow B_3/Z(B_3) \rightarrow Out(F_3)$ corresponding to the Artin's
action (\ref{artinaction})  of $B_3$ on $F_3$. 
\end{proof}
\begin{rem} Groups $Q^i_{i+1}, i=1,2$ are formed by one conic and arrangement of lines with no cycles and their structure was obtained in 
\cite{freedmangarber}.
\end{rem}

\section{Surfaces with $\Delta=0$: Scrolls $(F_e,\CO(E+kF))$}\label{scrollssection}

In this section we study the extremal subgroups in linear system $E+kF$ on the scroll $S_{k-e,k}=(F_e,\CO(E+kF))$ corresponding to 
saturated system $\nabla_k$. We give the upper bound on the number of equivariant isotopy classes pencils with reduced members
and calculate the corresponding extremal groups for several types of pencil, including all the types that appear in linear systems with the degree of 
discriminant less than 6.

\subsection{Pencils of hyperplane sections of scrolls.}

\begin{prop}\label{pencils classification} A pencil of hyperplane sections on the polarized surface $(F_e,\CO(E+kF)), k \ge e+1$ 
(i.e. a scroll $S_{k-e,k}$)
has $2k-e$ singular members counted with multiplicities. 
 A pencil satisfying the conditions (a),(b) of Prop. \ref{pencils} have the  singular members
 of the following types:

($\dagger$) A reducible curve that is a union of a smooth curve in linear system $\PP(H^0(E+(k-l)F)), 1 \le l \le k-e$ and  $l$ distinct rulings of the scroll,
i.e. a curve in the stratum $S_{1^l}$ of the dual variety of the scroll (cf. Theorem \ref{stratification}(a)).

($\dagger \dagger$) A reducible curve that is the union of the exceptional curve $E$ of $F_e$ (the directrix of the scroll) and a set of $k$ distinct lines of the ruling i.e.
a curve in the stratum $T_{1^k}$ (cf. Theorem \ref{stratification} (b)).

If a pencil has a fiber of type ($\dagger \dagger$) and $s$ fibers of type (a) having $l_i, i=1, \cdots s$ rulings in respective elements of the pencil then $\sum l_i=k-e$.
If a pencil has only fibers of type ($\dagger$), then the numbers of lines of ruling $l_i$ satisfy $\sum l_i=2k-e$. 

\end{prop} 

\begin{proof} Pencils on $S_{k-e,k} \subset \PP^{2k-e+1}$ with singular members satisfying the conditions of Prop.  \ref{pencils}  correspond to 
the lines in the dual projective space ${\PP^{2k-e+1}}^{\vee}$ and intersecting the dual variety of the scroll only in the points of the strata $S_{1^l}$ and at most one point in the strata $T_{1^k}$ (the latter is since out pencils have no fixed components). Indeed,  the curves in the 
remaining strata of the dual variety of $S_{k-e,k}$  as members of a pencil will have a component of multiplicity greater than 1 violating (a) in Prop. (\ref{pencils}). 
The descriptions ($\dagger$),($\dagger \dagger$) are just 
descriptions of the curves corresponding to the points in these strata. The relations for $l_i$ restricting the combinations of the strata that 
can appear in actual pencils follows from Theorem \ref{stratification} (c) and Bezout theorem. 
\end{proof} 

\begin{cor} Let $D$ be a curve of $F_e$ with classes of components having classes in the saturated system (cf. Section \ref{boundsonnumber}) $$\nabla(D_k) =\{E,F, E+iF, e \le i \le k\}.$$
 If there is surjection $\pi_1(F_e\setminus D)\rightarrow F_r$ where $r \ge max(2(k+1)-e,7)$.
 then either
 
 (A) $D$ is a union of $r+1$ curves each in the linear system $H^0(F_e,O(E+kF)$  one of which is 
  a curve as in ($\dagger \dagger$), $s$ curves of types $(\dagger)$ with $l_i, i=1, \cdots s$ being such that $\sum_1^s l_i \le k-e$
and the remaining $r-s$ curves being smooth. 

(B) $D$ is a union of $s$ curves of type ($\dagger$) with $l_i, 1\le k-e, i=1, \cdots s$ being such that $\sum_1^s l_i \le 2k-e$
and $r+1-s$ smooth curves in the linear system $H^0(F_e,\CO(E+kF))$.
\end{cor}

\begin{proof} Existence of the pencil follows from the calculation of threshold in Theorem \ref{keybound}. Cases (A) and (B) follow immediately from 
Prop. \ref{pencils classification} since they describe the singular elements of the pencil which may appear as components of $C$ with 
inequalities on $l_i$ coming from Bezout theorem and calculation of the multiplicities of the strata of dual variety of the scroll in Theorem \ref{stratification}.
\end{proof}

\begin{rem} We denote the groups $\pi_1(F_e\setminus D)$  of the complement to a curve in class (A) as $R_r^{l_1,\cdots l_s}$ ($\sum_1^s l_i \le k-e$) and 
for a curve in class (B) as $Q_r^{l_1,\cdots l_s}$. For a simplified notation $Q_r^i$ in the case when $k=e+1$, see Example \ref{casee+1}. Such 
notation reflects the number of singular curves of the type described by the strata containing the singular curves and the number of smooth 
members  among total $r+1$ deleted members of a pencil.
\end{rem}

\subsection{The number of extremal subgroups}

\begin{dfn} An $r+1$-marked pencil is a pencil with fixed set of $r+1$ of its members. 

(a) The combinatorial type of an $r+1$-marked pencil containing as marked members $s$ members in the strata $S_{1^{l_i}}, i=1,\cdots s, 1\le s \le r+1$
and remaining marked members being smooth, 
is the array $(l_i)$ viewed as a partition $\lambda$  of the integer $a= \sum_{i=1}^s l_i$ where  $0 \le a \le 2k-e, 1 \le l_i \le k-e$.  
Combinatorial type of an $r+1$-pencil having one of its singular elements in the stratum $T_{1^k}$ is the partition of $a$ (with $0 \le a \le k-e$) corresponding to the collection
of strata $S_{1^{l_i}}$ containing marked singular elements of this pencil and the remaining marked members being smooth. 

(b) Zariski weight of a combinatorial type given by a partition $\lambda$, is the number of isotopy classes of $r+1$-marked pencils preserving the combinatorial type of the pencil.
We denote by $Z(\lambda)\ge 0$ the Zariski weight of a combinatorial type given by partition $\lambda$.
$Z(\lambda)=0$ if and only if the set of strata $S_{1^{l_i}}, \lambda=(l_i)$ does not admit a multisecant intersecting them at a smooth points.
\end{dfn}

\begin{prop}\label{upperboundpartitions} 
(a) The upper bound on the number of combinatorial types of $r+1$-marked pencils on a scroll $(F_e,\CO(E+kF)$ for $r+1 \ge 2k-e$ is given by 
$$\sum_{j=0}^{2k-e} \lbrace p(2k-e-j,k-e)+p(k-e-j) \rbrace$$ 
where $p(n,k)$ is the cardinality of set of partitions $Par(n,k)$ of $n$ into parts of size at most $k$, $p(n)=p(n,n)$ is the number of partitions of $n$ and 
$p(0)=1$. 

(b) The number of $r$-extremal groups is not larger then the number classes of equivariant isotopy of $r+1$-marked pencil 
given by 
$$\sum_{j=0}^{2k-e}\sum_{\lambda \in Par(2k-e-j,k-e)} Z(\lambda)+\sum_{j=0}^{2k-e}\sum_{\lambda \in Par(k-j)} Z(\lambda)
$$
where the second sum corresponds to pencils containing stratum $T_{1^k}$ as one if singular members.
\end{prop} 

\begin{proof} 
We first count the partitions given by a combinatorial types of $r+1$-marked pencils with members only from $S$-strata. 
By Bezout relation for a pencil intersecting only strata $S_{1^l}$ in the discriminant having degree $2k-e$, one has $\sum l_i=2k-e, l_i \le k-e$ where the summation is over all 
singular members of the pencil. For the singular members which are marked, one has $0 \le \sum l_i=a \le 2k-e$ i.e. 
combinatorial type is given by a partition of $a=2k-e-j$.
For example, for $a=0$ all marked members are smooth and for $a=2k-e$, $a$ members have 2 irreducible components one of which is $F$ and
another is an irreducible curve in $E+(k-1)F$ and remaining $r+1-a$ marked members are smooth. 
The second summand represents the pencils with one singular member containing the directrix for which the Bezout relation is $\sum l_i=k-e$ 
since the multiplicity of smooth points in $T_{1^k}$ is $k$. The part (b) follows by the same argument as in (a) and the definition of Zariski weight. 
\end{proof} 

\begin{example}\label{casee+1} 
Let $k=e+1$. Then the above estimate becomes combinatorial $e+2$-marked pencils: $\sum_{j=0}^{e+2}p(e+2-j,1)+p(1-j)=e+5$. 
Indeed, the first sum counts the combinatorial types of pencils containing $a$ singular members in strata $S_{1^1}$ where $0\le a \le e+2$ and 
$e+2-a$ and each  summand is equal to 1. The second sum is 2 which accounts for two pencils: the one having one marked member from $T_{1^k}$, $k-e$ 
members from stratum $S_1$ and the rest are smooth and the second contains as marked members a member from $T_{1^k}$ and the rest are smooth.

The  reducible curves fundamental groups of the complements to which are extremal have as irreducible components $i=0, \cdots e+2$ curves in $E+eF$ and $i$ rulings and $r+1-i$ smooth curves in $E+(e+1)F$
For this $k$, a pencil satisfying condition (a) of Prop. 
\ref{pencils} 
cannot contain a curve in stratum $T$. Indeed, if this would be 
the case, then the base locus of the pencil of hyperplanes in $\PP^{e+3} \supset S_{1,e+1}$ must intersect the directrix at a base point of the pencil on 
$S_{1,e+1}$ and since none of the $E+eF$ components of 
the curves of the pencil intersects the directrix the ruling at this intersection point must be the fixed component i.e. we consider pencil which 
is not in $E+(e+1)F$ (after removing fixed component).  Hence Zariski weight of the pencils counted by the second sum in Proposition has $Z(\lambda)=0$ and hence  there are at most $e+3$ subgroups in corresponding to this linear system. One can show that this bound is sharp. 
\ref{upperboundpartitions} 

The pencils with members having components in the stratum $T$ do appear in linear systems $E+kF, k>e+1$ in which case the base point of the pencil on the directrix contains smooth irreducible components of singular members. 

We  denote the $r$-extremal group of the union of the $r+1$ members of a pencil on $S_{1,e+1}$ $i$ of which are reducible 
curves with two irreducible components in $E+eF$ and $F$ as $Q^i_{r}, e=0,\cdots e+2$.
\end{example} 

\begin{rem} 
Precise range of $Z(\lambda)$ i.e. cardinality of "Zariski multiplets" of reducible curves formed by members of 
pencils in linear systems on scrolls intersecting the same strata of the dual variety is not clear. 
However, the proof of Theorem  
\ref{extermalgroupsgenericcase}
  and Prop. 
  \ref{groupsTstratapencils} 
  below shows that the fundamental groups of such 
curves are isomorphic to the fundamental groups of the complements in a quadric to a union of rulings and the numbers of each of two ruling being determined by the strata the pencil intersects
\end{rem}

\subsection{Presentations of extremal groups of scrolls}

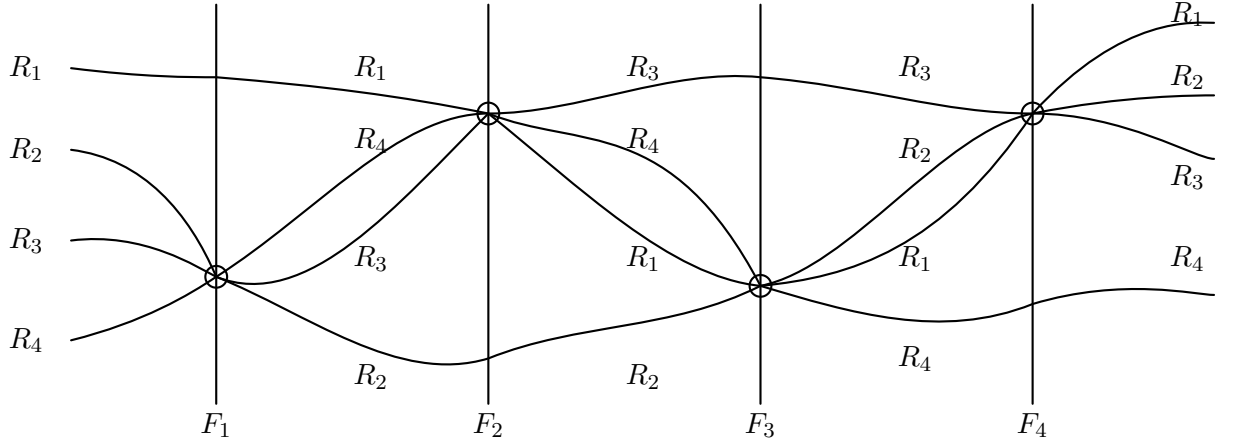
\begin{figure}\label{realpicture}
\centering
\hspace*{-1.0 cm}
\begin{tikzpicture}[scale=1.2,
    line cap=round,
    line join=round,
    every node/.style={font=\large}]

\path[use as bounding box] (-3,-2.5) rectangle (13,2.5);

\foreach \x/\lab in {0/{F_1},3/{F_2},6/{F_3},9/{F_4}}
{
  \draw[thick] (\x,-2.2) -- (\x,2.2);
  \node[below] at (\x,-2.2) {$\lab$};
}

\node[left] at (-1.8,1.5) {$R_1$};
\node[left] at (-1.8,0.6) {$R_2$};
\node[left] at (-1.8,-0.4) {$R_3$};
\node[left] at (-1.8,-1.5) {$R_4$};

\draw[thick]
(-1.6,1.5)
.. controls (-0.8,1.4) and (-0.3,1.4) ..
(0,1.4)
.. controls (1.2,1.3) and (2.0,1.2) ..
(3,1.0)
.. controls (4.0,1.0) and (5.0,1.5) ..
(6,1.4)
.. controls (7.2,1.3) and (8.0,1.0) ..
(9,1.0)
.. controls (10.0,2.1) and (10.8,2.0) ..
(11,2.0);

\node[left] at (2,1.5) {$R_1$};
\node[left] at (2,0.7) {$R_4$};
\node[left] at (2,-0.6) {$R_3$};
\node[left] at (2,-1.9) {$R_2$};


\draw[thick]
(-1.6,0.6)
.. controls (-0.8,0.5) and (-0.3,-0.1) ..
(0,-0.8)
.. controls (1.0,-1.2) and (2.0,0) ..
(3,1.0)
.. controls (4.0,0.6) and (5.0,1.0) ..
(6,-0.9)
.. controls (7.0,-0.7) and (8.0,0.8) ..
(9,1.0)
.. controls (10.0,1.2) and (10.8,1.2) ..
(11,1.2);

\node[left] at (5,1.5) {$R_3$};
\node[left] at (5,0.7) {$R_4$};
\node[left] at (5,-0.6) {$R_1$};
\node[left] at (5,-1.9) {$R_2$};


\draw[thick]
(-1.6,-0.4)
.. controls (-0.8,-0.3) and (-0.2,-0.7) ..
(0,-0.8)
.. controls (1.1,-0.1) and (2.0,1.0) ..
(3,1.0)
.. controls (4.0,0.2) and (5.0,-0.8) ..
(6,-0.9)
.. controls (7.0,-0.8) and (8.0,-0.5) ..
(9,1.0)
.. controls (10.0,1.0) and (10.8,0.5) ..
(11,0.5);

\node[left] at (8,1.5) {$R_3$};
\node[left] at (8,0.6) {$R_2$};
\node[left] at (8,-0.6) {$R_1$};
\node[left] at (8,-1.7) {$R_4$};


\draw[thick]
(-1.6,-1.5)
.. controls (-0.8,-1.3) and (-0.3,-1.0) ..
(0,-0.8)
.. controls (1.0,-1.2) and (2.0,-2) ..
(3,-1.7)
.. controls (4.0,-1.3) and (5.0,-1.4) ..
(6,-0.9)
.. controls (7.0,-1.2) and (8.0,-1.5) ..
(9,-1.1)
.. controls (10.0,-0.8) and (10.8,-1.0) ..
(11,-1.0);

\node[left] at (11,2.1) {$R_1$};
\node[left] at (11,1.4) {$R_2$};
\node[left] at (11,0.3) {$R_3$};
\node[left] at (11,-0.6) {$R_4$};

\draw[thick] (0,-0.8) circle (0.12);

\draw[thick] (3,1.0) circle (0.12);

\draw[thick] (6,-0.9) circle (0.12);

\draw[thick] (9,1.0) circle (0.12);

\end{tikzpicture}
\caption{Singular members $R_i+F_i, i=1,2,3,4$ of $E+3F$ on $F_2$}
\end{figure}

The main observation about the complements to  unions of members of pencils of hyperplane sections of a scroll with reduced members 
(cf. Figure \ref{realpicture}) 
is that they are biregular to the complements to the unions of fibers to two fibrations of $\PP^1\times \PP^1$ and in particular all extremal groups are products of free groups. 
In this section we will describe a case by case analysis of the complements to such pencils. 

\begin{thm}\label{extermalgroupsgenericcase} a) The fundamental group of the complement to the union of $r+1$ members of a generic pencil of hyperplane sections on a polarized 
surface $(F_e,\CO(E+kF)), k\ge e+1$ including all $2k-e$ singular members of this pencil is a product of two free groups 
\begin{equation}\label{productfreegroups} F_{2k-e-1}\times F_{r}
\end{equation}

b) On a polarized surface as in a), the fundamental group of the complement to a curve that is a union of $r+1$ members of a non-generic pencil intersecting at most two strata $S_{1^i}$ with $i>1$ \footnote{we restrict ourself only to such combinations of strata since only these cases appear in classification of 
pencils in linear systems with degree of discriminant t most 6. The method we use works for more general pencils} and such that irreducible components of all singular members of this pencil are among the components of the curve 
 is also given by (\ref{productfreegroups}). 
\end{thm}
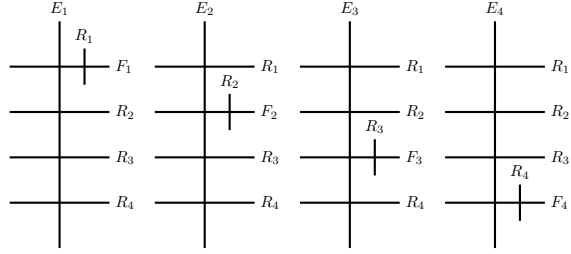
\begin{figure}[htbp]
\centering
\begin{tikzpicture}[scale=0.6, >=stealth, every node/.style={transform shape}]

\foreach \i [evaluate=\i as \xshift using (\i-1)*3.2] in {1,2,3,4}{
  \begin{scope}[xshift=\xshift cm]
    
    \draw[thick] (0,0) -- (0,5) node[above] {$E_{\i}$};

    \draw[thick] (-1.1,4) -- (1.1,4) node[right] {\ifnum\i=1 $F_1$\else $R_1$\fi};
    
    \draw[thick] (-1.1,3) -- (1.1,3) node[right] {\ifnum\i=2 $F_2$\else $R_2$\fi};
    
    \draw[thick] (-1.1,2) -- (1.1,2) node[right] {\ifnum\i=3 $F_3$\else $R_3$\fi};
    
    \draw[thick] (-1.1,1) -- (1.1,1) node[right] {\ifnum\i=4 $F_4$\else $R_4$\fi};

    \pgfmathtruncatemacro{\ypos}{5-\i}
    \draw[thick] (0.55,\ypos-0.4) -- (0.55,\ypos+0.4) node[above] {$R_{\i}$};

  \end{scope}
}

\end{tikzpicture}
\caption{Blow up of $F_2$ at base points of generic pencil in $E+3F$}
\label{fig:four_graphs}
\end{figure}

\begin{proof} First consider the case a). Recall that in a generic pencil each singular curve is reducible with two components one of which is a ruling. $2k-e$ base points of this pencil with one lying on each of $2k-e$ rulings (that are irreducible components of singular 
members of this pencil) and through each of such point all $2k-e-1$ remaining residual irreducible components of singular members are passing 
as well as $r+1-2k-e$ smooth deleted members (cf. Fig. \ref{realpicture} for the case $e=2,k=3$; circles mark $4$ base points of the pencil). Blowing up each of base points produces $2(2k-e)+(r+1-2k+e)+(2k-e)=2(2k-e)+r+1$ irreducible curves. 
Self-intersection of proper preimage of each of smooth component in $E+kF$ is zero. On the other hand self-intersection of a proper preimage of 
residual component of a singular member, having on the scroll self-intersection $2k-e-2$ and containing $2k-e-1$ base points will be $-1$ 
as will be the self intersections of irreducible components that are rulings (and of course the exceptional curves). Consider 
blow down of proper preimages of the fibers. The images of proper transforms of residual curves, of the exceptional curves will have self-intersection zero since each intersects proper transform of single fiber and so will be the self-intersections of proper pre-images of deleted $r+1-2k+e$ members of 
smooth members of the pencil. Each smooth member and residual component of singular member will intersect image of exceptional curve at single point. 
This implies that the intersection form on the blow down is  $\begin{pmatrix} 0 & 1 \\ 1 & 0 \\ \end{pmatrix}$ the surface itself is $\PP^1 \times \PP^1$. Hence the complement to $r+1$ member of the pencil we consider is biregular to 
$\PP^1\times \PP^1$ with deleted $r+1$ and $2k-e$ in corresponding rulings of the surface.   

In the case b) configurations and the numbers of irreducible components will be different than in a). 
Let us start with considering in more detail the case when one has only one member in 
the stratum $S_{1^i}$ (so that remaining deleted members are $2k-e-i$ members from the main stratum $S_1$ and remaining $r+1-(2k-e-i+1)=r-2k+e+i$ members being smooth). We have $2(2k-e-i)+(i+1)=2(2k-e)-i+1$ irreducible components given by $2k-e$ rulings, $2k-e-i$ residual components of the members with 
1 singular points, 1 residual component of the member with $i$ singularities and the rest being smooth members. 
As before, all $2k-e$ base points positioned by one on each ruling, but multiplicity of the union of all components at each base point is $2k-e-i+1$. 
Self-intersection of the residual component  $Res_0$ of the member from $S_{1^i}$ is $2(k-i)-e$ while for the members in $S_1$ are the same as in previous case. After blow up of all $2k-e$ base points, the self intersection of the proper transform of $Res_0$ is $-i$ and since $Res_0$ intersects $i$ rulings its image after blow down of the rulings will be zero. So will be the case for the residual curves of the members $S_1$  and the exceptional 
curves of the blow ups of the scroll. As a result we obtain, that the the complement to original arrangement of irreducible curves is biregular to 
$\PP^1 \times \PP^1$ with again removing respectively $2k-e$ and $r+1$ rulings from different rulings.   

If a pencil contains $2k-e$ members from the strata $S_{1^i},S_{1^j},S_1$  then the number of irreducible components will be $3(2k-e)-i-j+2$ with the number of base points unchanged.
Keeping track of the multiplicities of base points and changes in blow ups and blow down shows the claim of the theorem.
\end{proof}

\begin{cor}\label{addingmembersgenericcase} The fundamental group of the complement to the union of $r+1$ members of the generic pencil on polarized surface
$(F_e,\CO(E+kF)), k\ge e+1$ among which there are  $2k-e-i$ singular members and $r+1-2k+e+i$ smooth members ($0 \le i \le 2k-e$) is a product:
$$F_{r} \times F_{2k-e-1-i}$$
\end{cor} 

\begin{proof} To calculate the fundamental group of the complement to the union of curves as in the corollary, apply the Theorem  \ref{extermalgroupsgenericcase} to the case of generic pencil with $r+1+i$ members removed which include {\it all} singular members of this pencil. 
The fundamental group of the complement in $F_e$ to this union of $r+i+1$ curves, or equivalently, the complement to their total 
preimages in the blow up of the base points, will be $F_{r+i} \times F_{2k-e-1}$. Now let us glue in $i$ singular members. The result will be the quotient 
of the latter groups by the normal subgroup generated by meridians of residual curves and the rulings of glued in members. Meridians of $i$ residual 
curves correspond to $i$ generators of the first factor and the meridians of the ruling of glued in singular members will 
correspond after blow down of these $i$ rulings to $i$ meridians of exceptional curves on resolution of the $i$ (out of total $2k-e$) base points of the original pencil corresponding to $i$ glued in members. These meridians of exceptional curves correspond to  
the  generators in the free group $F_{2k-e-1}$ of the second factor. Clearly, the quotient by the normal subgroup generated by such relations 
\footnote{i.e. one can see the relations $m_E=1$, where $m_E$ is the meridian of exceptional curve $E$ obtained by blowing up of the ruling of 
glued in singular member of the pencil, as follows. After blowing up a base point, the ruling $F$  containing with satisfy $F^2=-1$ and small $C^{\infty}$ 
deformation of it will be $S^2$ intersecting the proper preimage $F$, the exceptional curve $E$ and preimage of the residual curve $R$ of the singular fiber. Their meridians
will satisfy the relation $m_F=m_E\cdot m_R$. After blowing down the proper preimage of the ruling and adding relations $m_F=m_R=1$ we also will have $m_E=1$}
is isomorphic to the product given in the statement of the corollary.  
\end{proof}. 

\begin{prop}\label{groupsTstratapencils} The fundamental groups of the complement to arrangements $\A^{r+1}_S(T, \S_1\cdots \S_{k-1})$ of $r+1$ curves on the scrolls $S=(F_1,\CO(E+2F)), (F_1,\CO(E+3F)),(F_2,\CO(E+3F))$ that is a union of $r+1$ members of a pencil of hyperplane sections, all of which are smooth except for $k-1$ curves 
from the strata $\S_i, i=1,\cdots k-1$ and one is from a stratum  of type $T$ are as follows:

a) for strata $(T_2,S_1)$ on $(F_1,\CO(E+2F))$, one has $$\pi_1(F_1\setminus \A^{r+1}_{F_1,\CO(E+2F)}
(T_2,S_1))=F_{r}\times F_{2}$$

b) for $(T_{1^3}, S_1^2)$ on  $(F_1,\CO(E+3F))$ one has 
$$\pi_1(F_1\setminus \A^{r+1}_{F_1,\CO(E+3F)}(T_{1^3},S^2_1))=F_{r}\times F_4$$

c) for  $(T_{1^3}, S_{1^2})$ on  $(F_1,\CO(E+3F))$ one has $$\pi_1(F_1\setminus \A^{r+1}_{F_1,\CO(E+3F)}(T_{1^3},S_{1^2}))=F_{r} \times F_4$$

d) for  $(T_{1^3}, S_{1})$ on  $(F_2,\CO(E+3F))$ one has $$\pi_1(F_2\setminus \A_{F_2,\CO(E+3F)} (T_{1^3},S_{1}))=F_{r}\times F_3$$
\end{prop} 

\begin{proof} The arguments for these 4 cases are the same and consist of showing that the complement to $r+1$ members
in all cases are biholomorphic to products of complements in $\PP^1$ to finite sets of points. We spell out the details 
only in two cases.
  
In the case a) we have a pencil with 2 its singular members deleted, one having as irreducible components the directrix $E$ and rulings 
$F_1,F_2$, and another being a union of a residual curve $R$ with class $E+F$ and ruling $F_3$, and also deleted $r-1$ smooth members $C_j, j=1, \cdots r-1$. 
This pencil has $3$ base points, $R\cap F_1,R\cap F_2, E\cap F_3$. Blowing them up gives a rational surface with middle Betti number $b_2=5$ with exceptional curve 
$E_i, i=1,2,3$  on which 
one has self-intersections of the proper preimages $\bar R, \bar E, \bar F_i, \bar C_j, i=1,2,3, j=1,\cdots r-1$ as follows  $\bar R^2=-1, \bar E^2=-2, 
\bar F_i^2=-1, \bar C_j^2=0$. After blowing down 3 smooth rational curves $\bar F_i$, one obtains the surface with the images 
of $\bar R, \bar E, \bar C_j, E_i$  having the self-intersection zero, while images of $E_i$ will intersect the images $C_i, R, E$ transversally.
This implies that after blow down the surface we obtain is biregular to  $\PP^1\times \PP^1$ and the sets of curves $E_i, i=1,\cdots 3$ and $C_i,R,E$ are distinct rulings of this product. 
Hence the complement to the above $r+1$ members of this pencil is biholomorphic to $\PP^1\setminus [3] \times \PP^1\setminus [r+1]$ and the claim follows.

Now consider the case c) i.e. when among deleted members of the pencil, there are two singular members and one 
is the union of the directrix $E$ and 3 distinct rulings $F_1,F_2, F_3$, while another is a section $R$ from $E+F$ and two distinct rulings $F_4,F_5$
also different from the rulings in the first singular fiber. The 5 base points are given by the intersection points $RF_1,RF_2,RF_3$ and $EF_4,EF_5$.
Blowing them up gives rational surface with $b_2=7$, containing proper preimages  $\bar E, \bar F_i, \bar R, i=1, \cdots 5$ of the corresponding curves on the scroll, and additional 5 exceptional curves $E_1,\cdots E_5$. One has $\bar F_i^2=-1, \bar R^2=-2, \bar E^2=-3$.
Blowing down proper preimages of the rulings $\bar F_i$ gives us a rational surface with $b_2=2$, on which the intersection indices of images are $\bar R^2=\bar E^2= \bar E_i^2=0$ 
which $\bar E \bar E_i=1, RE_i=1, i=1,\cdots 5 $. Hence the resulting surface is $V_2=\PP^1\times \PP^1$ with $E,R$  and $E_i,i=1,\cdots 5$ 
are different rulings of this direct product.  Moreover, proper preimage  $\bar C_i, i=1, \cdots r-1$ of a smooth member $C_i$ on blow up at five base point will have 
selfintersection 0, and contraction of proper preimages of ruling does not change it as well as transversal intersections with $E_i$. Hence the images 
of $\bar C_i$  they become 
rulings of $V_2$ linearly equivalent to the images of $\bar E$ or $\bar R$. This implies c). 

In remaining cases, the same argument shows that complement in the case b) (when the pencil has 3 singular members and the five base points
located at the 3 intersections of residual curve of 2 members from stratum $S_1$ and the ruling of the member from $T_{1^3}$ and two intersections
of the directrix with one of residual  curves and the ruling of another for each of two curves from the stratum $S_1$) is biholomorphic to 
$\PP^1\setminus [5] \times \PP^1\setminus [r+1]$ (5 rulings in the first factor are the images of 5 exceptional curves of blow up of $F_1$ at 
5 base points of $E+3F$.

In the case d)  the complement is biholomorphic to $\PP^1\setminus [4] \times \PP^1\setminus [r+1]$ with 4 rulings removed in the first factor
corresponding to 4 base points of $E+3F$ in the scroll $F_2$. 
\end{proof}

Next corollary describes fundamental groups of the complement to union of members of pencils which 
contain a proper (possibly empty) subset of the full collection of the singular members of the pencils 
 considered in Prop.\ref{groupsTstratapencils}. If $\S_1,\cdots, \S_k$ are the strata containing 
 elements of the pencil, one for each singular member of the pencil, we decorate the notation in Prop.\ref{groupsTstratapencils}
  denoting by $\A^{r+1}(\S_1^{\pm},\cdots \S_k^{\pm})$ the type of the reducible curves, where the superscript $+$ (resp. $-$) indicates
 that the curve in corresponding stratum is a deleted member (resp. not a deleted member) of the pencil. 

\begin{cor} The fundamental group of the complement:
$$\pi_1(F_e \setminus \A^{r+1}(\S_1^{\pm},\cdots \S_k^{\pm})=F_{r}\times F_{2k-e-i}
$$
where $i$ is the number of rulings appearing in the curves of the strata $\S_i$ with superscript $-$.
\end{cor}

\begin{proof} The argument is the same as in the proof of the Corollary \ref{addingmembersgenericcase} combined with the observation 
that filling in member from stratum $\S_i$ which member contain $r$ rulings results in taking quotient by the normal 
closure of meridians of $r$ rulings which after birational transformation described in Prop. \ref{groupsTstratapencils} results in taking normal closure in $F_{N-1}\times F_{2k-e}$ of meridians of $r$ exceptional curves of blow down in this birational transformation. 
\end{proof} 

\section{Scrolls with the degree of discriminant less than 6}\label{finalsection}

In this section we describe the classification of essential and transversal pencils on polarized simply connected surfaces  
with the degree of discriminant less than 6. We also find explicitly the corresponding 
 $r$-extremal fundamental groups.

\begin{thm}\label{finalsummary} If the fundamental group of the complement to a curve on a simply connected surface 
admits a surjection on a free group of rank $r$ greater than 2 that is essential and transversal then the curve is a union of members of a pencil 
and if the pencil is contained in a very ample linear system with the degree of discriminant less than 6 and if $r > 6$ then 
the surface, the pencil and the corresponding the fundamental group of the complement to the curve are listed in the table below. 
\end{thm} 

{\bf Notation}
\begin{itemize}

\item 

Column 1: the polarized surface:

\item Column 2: The degree of the surface, the degree of the discriminant  and threshold $t$ 
such that for $r>t$ a surjection onto a free group induced by a pencil implies the 
isomorphism type of the group (under conditions of the Theorem). 

\item Column 3: Strata of the discriminant intersected by the  pencil  i.e. the structure of the singular members of the pencil.

\item Column 4: Extremal groups and a reference to  the proof.
\end{itemize}

Theorem \ref{finalsummary} allows to conclude that the fundamental group to a reducible curve with 
classes of components in a saturated set $\nabla$ and admitting a surjection onto a free group of
a rank $r$ above the threshold, then this group admits a surjection onto one of the extremal groups corresponding to $\nabla$.
Moreover, additional conditions on the preimages of conjugacy classes of generators of the free group allow specify 
the extremal group through which the surjection factors. Here we give a sample of exact statements for several extremal groups.

\begin{cor} Let $\A$ be a conic line arrangement in $\PP^2$ such that there exist a surjection 
$\phi: \pi_1(\PP^2\setminus \A) \rightarrow F_r$. Assume that the image of a conjugacy 
class of each of $N$ irreducible components of $\A$ is either trivial or is the conjugacy class of one of the generators $x_i,i=1,\cdots r$
of $F_r$ or is the conjugacy class of $x_{r+1}=(x_1\dots x_r)^{-1}$. Assume also that there are $r+4$ conjugacy classes of components of $\A$ that have
as images $r+1$ allowed conjugacy classes in $F_r$ and that among these $r+4$ classes, 6 are pairwise mapped to 
3 conjugacy classes $\{x_{i_1}\},\{x_{i_2}\},\{x_{i_3}\}$ in $F_r$ and the remaining $r-2$ are mapped one-to-one to distinct conjugacy classes in $F_r$
$x_i, i \ne i_1,i_2,i_3$.
Then $\phi$ factors as \footnote{groups $Q^i_r$ discussed in Theorem \ref{presentationquadrics}} 
\begin{equation}\label{factorization} \pi_1(\PP^2\setminus \A) \rightarrow Q^3_r \rightarrow F_r
\end{equation}

If instead of $r+4$ conjugacy classes, there are $r+2$ conjugacy classes of meridians of components of $\A$, among which $4$ pairwise 
are mapped to $2$ conjugacy classes $\{x_{i_1}\},\{x_{i_2}\}$ of generators and remaining $r-2$ are mapped to $x_i, i\ne i_1,i_2$ one to one, 
then $\phi$ factors through $Q^2_r$. Similar factorizations hold for remaining extremal groups $Q^i_r$.
\end{cor}

\begin{proof} Let $\A'$ be the subarrangement of $\A$ formed by $r+4$ components having non-trivial image in $F_r$. $\pi_1(\PP^2\setminus \A')$ has 
surjection onto $F_r$ and the conditions on the meridians
of these components and classification of the pencils of quadrics above  imply that this is arrangement formed by three singular quadric and $r-2$ smooth quadrics in a generic pencil of quadrics. Hence $\pi_1(\PP^2\setminus \A')=Q^3_r$ and factorization (\ref{factorization}) is induced by inclusion $\PP^2\setminus A \subset \PP^2\setminus \A'$.
The proof of the last claim is identical to the case of $Q^3_r$
\end{proof}

\begin{center}
\begin{table}\label{resultstable}   
\centering
\caption{Extremal groups for surfaces with degree of the dual variety less than 6}
\begin{tabular}{|l | c | r | r | } 
  \hline
  $(X,L)$  & $L^2$/ deg Disc   & Pencils/Singular members  & Extremal Groups \\ \hline
  $(\PP^2,\CO(1))$ & $d=1, D=0, r>3 $  & None & $F_r$ \\ \hline 
  $ (\PP^2,\CO(2))$ & $d=4, D=3, r>5$ & $\{ L+L,L+L,L+L\}$ & {\it Theorem \ref{presentationquadrics}}: $Q^i_r, i=0,1,2,3$
  \\ \hline
   $ (F_0,\CO(E+F))$ & $d=2, D=2, r>6$ & $E+F,E+F $ & {\it Prop. \ref{quadricextremal}} $\ZZ\times F_{r}, F_{r} $
        \\ \hline
     $ (F_0, \CO(E+2F)$ & $d=4, D=4, r>6$ & $S_1^4: \{\begin{matrix}((E+F)+F, \cr (E+F)+F, \cr (E+F)+F, \cr (E+F)+F \end{matrix} \}  $ &
  $ \begin{matrix} {\it Cor} \ \ref{addingmembersgenericcase}: F_3\times F_r, F_2\times F_r, \cr  \ZZ \times F_r,  F_r  \cr \end{matrix}$ 
    \\ \hline
     $ (F_0, \CO(E+2F)$ & $d=4, D=4, r>6$ & $S_{1^2}= T_{1^2}, S_1^2: \{\begin{matrix}((E+F+F, \cr (E+F)+F, \cr (E+F)+F, \cr \end{matrix} \}  $ &
  $ \begin{matrix} {\it Cor} \ \ref{addingmembersgenericcase}: F_3\times F_r, F_2\times F_r, \cr  \ZZ \times F_r,  F_r  \cr \end{matrix}$ 
       \\ \hline
     $ (F_0, \CO(E+2F)$ & $d=4, D=4, r>6$ & $S_{1^2}^2: \{\begin{matrix}((E+F+F, \cr (E+F+F \cr \end{matrix} \}  $ &
  $ \begin{matrix} {\it Cor} \ \ref{addingmembersgenericcase}: F_3\times F_r, F_2\times F_r, \cr  \ZZ \times F_r,  F_r  \cr \end{matrix}$ 
    \\ \hline
    $ (F_1,\CO(E+2F))$ & $d=3, D=3, r>6$ & $S_1^3: \{\begin{matrix}(E+F)+F, \cr (E+F)+F, \cr (E+F)+F \end{matrix}\}   $ & 
    {\it Cor.}  \ref{addingmembersgenericcase}: 
    $F_2\times F_r, \ZZ\times F_r, F_r$ \\ \hline
    $ (F_1,\CO(E+2F))$ & $d=3, D=3, r>6$ & $T_2,S_1: \{\begin{matrix} (E)+F+F \cr 
                                                                                           (E+F)+F\end{matrix} $ & {\it Prop.} \ref{groupsTstratapencils} $F_2\times F_r, \ZZ\times F_r, F_r$
      \\ \hline
     $ (F_1, \CO(E+3F)$ & $d=5, D=5, r>6$ & $S_1^5: \{\begin{matrix}(E+2F)+F, \cr (E+2F)+F, \cr (E+2F)+F, \cr (E+2F)+F, \cr (E+2F)+F \end{matrix} \}  $ &
  $ \begin{matrix} {\it Cor} \ \ref{addingmembersgenericcase}: F_4\times F_r, F_3\times F_r, \cr  F_2\times F_r, \ZZ \times F_r, F_r  \cr \end{matrix}$ \\ \hline
      $ (F_1, \CO(E+3F)$ & $d=5, D=5, r>6$ & $S_{1^2},S_1^3:\{\begin{matrix}(E+F)+F+F, \cr (E+2F)+F, \cr (E+2F)+F, \cr (E+2F)+F \cr \end{matrix} \}  $ & 
     $ \begin{matrix} {\it Cor} \ \ref{addingmembersgenericcase}: F_4\times F_r, F_3\times F_r, \cr  F_2\times F_r, \ZZ \times F_r, F_r  \cr \end{matrix}$
      \\ \hline
       $ (F_1, \CO(E+3F)$ & $d=5, D=5, r>6$ & $S_{1^2}^2,S_1\{\begin{matrix}( (E+F)+F+F, \cr (E+F)+F+F, \cr (E+F)+F, \cr  \end{matrix} \}  $ &
       $ \begin{matrix} {\it Cor} \ \ref{addingmembersgenericcase}: F_4\times F_r, F_3\times F_r, \cr  F_2\times F_r, \ZZ \times F_r, F_r  \cr \end{matrix}$ 
        \\ \hline
        $ (F_1, \CO(E+3F)$ & $d=5, D=5, r>6$ & $T_{1^3},S_1^2\{\begin{matrix}(E)+F+F+F, \cr (E+2F)+F, \cr (E+2F)+F, \end{matrix} \}  $ & 
   $\begin{matrix}{\it Prop.} \ref{groupsTstratapencils}: F_4\times F_r,  \cr F_3\times F_r, F_2\times F_r, \ZZ \times F_r, F_r \cr \end{matrix} $    \\ \hline
         $ (F_1, \CO(E+3F)$ & $d=5, D=5, r>6$ & $T_{1^3},S_{1^2}:\{\begin{matrix}(E)+F+F+F, \cr (E+F)+F+F, \end{matrix} \}  $ &
         $\begin{matrix}{\it Prop.} \ref{groupsTstratapencils}: F_4\times F_r,  \cr F_3\times F_r, F_2\times F_r, \ZZ \times F_r, F_r\cr \end{matrix} $
          \\ \hline
      $ (F_2,\CO(E+3F))$ & $d=4, D=4, r>6$ & $S_1^4:  \{\begin{matrix} (E+2F)+F, \cr (E+2F)+F, \cr (E+2F)+F, \cr (E+2F)+F \end{matrix} \}  $ 
       &  $ \begin{matrix} {\it Cor} \ \ref{addingmembersgenericcase}: F_3\times F_r, \cr  F_2\times F_r, \ZZ \times F_r, F_r  \cr \end{matrix}$
       \\ \hline
       $ (F_2,\CO(E+3F))$ & $d=4, D=4, r>6$ & $T_{1^3},S_1  \{\begin{matrix}(E)+F+F+F, \cr (E+2F)+F, \cr  \end{matrix} \} $ 
       & 
       $\begin{matrix}{\it Prop.} \ref{groupsTstratapencils}: F_3\times F_r,  \cr F_2\times F_r, \ZZ \times F_r, F_r \cr \end{matrix} $
       \\ \hline
           $ (F_3,\CO(E+4F))$ & $d=5, D=5, r>6$ & $S_1^5:  \{\begin{matrix} (E+3F)+F, \cr (E+3F)+F, \cr (E+3F)+F, \cr (E+3F)+F \cr E+3F \end{matrix} \}  $ 
       &  $ \begin{matrix} {\it Cor} \ \ref{addingmembersgenericcase}: F_3\times F_r, \cr  F_2\times F_r, \ZZ \times F_r, F_r  \cr \end{matrix}$
       \\ \hline
 $ (F_3,\CO(E+4F))$ & $d=5, D=5, r>6$ & $T_{1^4},S_1  \{\begin{matrix}(E)+F+F+F+F, \cr (E+3F)+F, \cr  \end{matrix} \} $ 
       & 
       $\begin{matrix}{\it Prop.} \ref{groupsTstratapencils}: F_4\times F_r,  F_3 \times F_r \cr F_2\times F_r, \ZZ \times F_r, F_r \cr \end{matrix} $
       \\ \hline
\end{tabular}
\end{table}
\end{center}

\begin{rem}
 There are 3 scrolls with the degree of discriminant equal to 6: $S_{3,3},S_{2,4},S_{1,5}$.
One has $p(6,3)=8, p(3)=3, p(6,2)=4,p(2)=2, p(5,1)=1,p(1)=1$. Hence one has 11 combinatorial types 
of pencils, for $S_{3,3}$ (respectively 6 types for $S_{2,4}$ and 2 types for $S_{1,5}$. The extremal groups
for $S_{2,4}$ and $S_{1,5}$ follows from Proposition  \ref{groupsTstratapencils}  and Corollary  \ref{addingmembersgenericcase}.
For $S_{3,3}$ one has the cases of pencils 
with $3$ members from the strata $S_{1^2}$ and not included in Proposition  \ref{groupsTstratapencils}.
 The extremal groups in this cases can be 
obtained using the method in Appendix and will be addressed elsewhere. 
\end{rem}


\section{Appendix: Zariski-van Kampen theorem for complements to divisors on smooth simply connected surfaces via pencils of curves}

Consider a quadruple $(\tilde X,G,\pi_L,E)$, where $\tilde X$ is a smooth simply connected projective surface, $G$ is a reduced curve on $\tilde X$ 
(possibly empty), $\pi_L: \tilde X \rightarrow \PP^1$ is a regular map such that all fibers are reduced and $E$ is a section of $\pi_L$ with none of the fibers $\pi^{-1}(p), p\in \PP^1$ has a singularity at its intersection point with $E$. Not that these assumptions imply that all fiber of $\pi_L$ are transversal to $E$.
  We assume also that $G\cap E=\emptyset$. Cases that arise in this paper are the following: firstly, the blow-up $\tilde X$ of all base points 
of an essential and transversal pencil $L$ (cf. Definition \ref{threshold}) 
on a surface $X$, $G$ is a curve not containing any base points of $L$ and $E$ 
is one of the exceptional curves of the blow up.  Second, there is the case of Hirzebruch surface $F_e$, $L$ is the pencil of rulings and $E$ is the section of the ruled surface $F_e$ such that  $E^2=-e$. 

We describe a presentation for $\pi_1(\tilde X\setminus G)$ in terms of the appropriate geometric monodromy action. 
Note that a far-reaching generalization of Zariski-van Kampen theorem due to I. Shimada (cf. Theorem 3.20 in 
\cite{shimada}) is not applicable to our situation since the auxiliary subspace $Z\subset \PP^1$, appearing  in condition (Z) in this work, admitting 
surjection $\pi_2(Z) \rightarrow \pi_2(\PP^1)$ must be $\PP^1$ itself, but in our situation we do not have a section 
of $\pi_L$ that would allow a choice of base point $\tilde b \in \tilde X\setminus E$  needed to define the action on $\pi_1(\pi_L^{-1}(\tilde b)$ (rather we have two steps: calculation of presentation for  $\pi_1(\tilde X\setminus G\cup \pi_L^{-1}(p)),p \in \PP^1$ using restriction of $\pi_L$ 
admitting a section and additional relations resulting in gluing back $\pi_L^{-1}(p)$). Moreover, we do not assume that the fibers of $\pi_L$ are irreducible.

Let $\widetilde{E}$
be a small $C^{\infty}$ deformation $\widetilde {E}$ such that $\tilde E$ is a section of the normal bundle of $E$ and $\widetilde{E} \cap E$ is a single point $B'$ (or empty if $E^2=0$) 
and the order of tangency of $\widetilde{E}$ and $E$ is $\vert (E\cdot E) \vert$, and let $\bar B'=\pi_L(B')$.  
Over $\PP^1\setminus \bar B'$ $\tilde E$ is a non-vanishing section which 
can be used to trivialize the normal bundle to $E$ in $\pi_L^{-1}(\PP^1\setminus \bar B')$. We use
$\pi_L^{-1}(d_0) \cap \tilde E, d_0 \in \PP^1\setminus (Crit \cup \bar B')$ as a base point of the smooth closed surface $\pi_L^{-1}(d_0)$ 
Without loss of generality one can assume that the fiber $\pi_L^{-1}(\bar B')$ is smooth. We call it the {\it fiber at infinity}.
Let $Crit \subset \PP^1$ be the set of critical values of $\pi_L$ and $\pi_L\vert_G$ i.e. 
$$Crit=\{p \in \PP^1 \vert \pi_L^{-1}(p) \ {\rm or} \ \pi_L^{-1}(p)\cap G \ {\rm is} \ {\rm singular} 
\footnote{i.e. the local intersection index is greater than one at one of the intersection points.}
\} $$
and let $\bar B'' \subset \PP^1\setminus Crit \setminus \bar B')$ be a finite, possibly empty,  set. Fibers over points in $B''$ are smooth and may be deleted 
in the case we will  need to consider complement to a union of $G$ and several fibers of $\pi_L$.
The curve $G$ provides a marking of the fiber $\pi_L^{-1}(p), p \in \PP^1\setminus (Crit \cup \bar B' \cup \bar B'')$ by points $\pi_L^{-1}(p)\cap G$ and we view $E \cap \pi_L^{-1}(p))$ 
as a puncture endowed with the normal to $E$ unit vector $v(p) \in T_{\pi_L^{-1}(p)\cap E}(X)$  that is the tangent to $\pi_L^{-1}(p)$ at the puncture. 
Let $d_0 \in \PP^1 \setminus Crit \cup \bar B' \cup \bar B''$ be a point that will serve as the base point for the loops in $\PP^1$.

Let $\beta$ denote the geometric monodromy of the quadruple $(\tilde X,G,\pi_L,E)$
 defined as the homomorphism given by horizontal arrow in (\ref{geommonodromyformula}) with the target being the mapping class group of diffeomorphisms of the surface $\pi_L^{-1}(d_0)$, with punctures 
 \footnote{or markings; we follows conventions in  \cite{FM} section 2.1} 
  at $\pi_L^{-1}(d_0) \cap G$,  
and also at  $\pi^{-1}_L(d_0) \cap E$, in addition endowed with a vector  $v(\pi_L^{-1}(d_0) \cap \tilde E)$ in the unit tangent bundle of the punctured surface $\pi^{-1}_L (d_0) \setminus (E \cap G) $
(cf. \cite{FM}, Section 4.2.5):
\begin{equation}\label{geommonodromyformula} 
\begin{matrix}\beta: \pi_1(\PP^1\setminus (Crit \cup \bar B' \cup \bar B'',d_0)) & \rightarrow & Mod(\pi_L^{-1}(d_0),\pi_L^{-1}(d_0)\cap G,v(\pi_L^{-1}(d_0)\cap \tilde E)) \cr 
& \searrow & \downarrow \cr 
& & Mod(\pi_L^{-1}(d_0) \setminus T(E) \cap\pi_L^{-1}(d_0)
,\pi_L^{-1}(d_0)\cap G,\partial T(E\cap\pi_L^{-1}(d_0)) \cr
\end{matrix}
\end{equation} Recall that it is 
given by assigning to a path $\gamma$ the isotopy class of 
diffeomorphism of the triple $(\pi_L^{-1}(d_0),\pi_L^{-1}(d_0)\cap G,v(\pi_L^{-1}(d_0)\cap \tilde E))$ 
obtained by selecting a trivialization of the locally trivial fibration  of the triple
$(\phi_{\gamma}^{-1})^*( \pi^{-1}(\gamma), G\cap \pi^{-1}(\gamma),v(\pi_L^{-1} \cap \tilde E)) \rightarrow I$ preserving the trivialization of the normal bundle to $E$ over $\gamma$  where $\phi_{\gamma}: I \rightarrow \gamma$ is a continuous map from $I$ onto $\gamma$  
such that $\phi_{\gamma}(0)=\phi_{\gamma}(1)=d_0$.
The slanted arrow gives an alternative description of the target of geometric monodromy using the identification of the mapping class group of diffeomorphisms of marked 
Riemann surface fixing point-wise the boundary of a small disk, centered at the puncture where $v$ is tangent to $\pi^{-1}_L(d_0)$ and the target of the 
horizontal arrow  
$ Mod(\pi_L^{-1}(d_0),\pi_L^{-1}(d_0)\cap G,\partial T(E)\cap\pi_L^{-1}(d_0))$ (cf. \cite{FM}, sect. 4.2.5).
Hence the above construction gives geometric monodromy also with values in the mapping class group of a surface with boundary. In particular, the target of $\beta$ acts on the fundamental group of $\pi_1(\pi_L^{-1}(d_0) \setminus (\pi_L^{-1}(d_0)\cap (G \cup T(E)), p)$ with a base point $p \in   \partial(\pi_L^{-1}(d_0)\cap T(E))$
\footnote{this is a free group}.  

\begin{thm}\label{ZvK} Let $(\tilde X,G,\pi_L,E)$ be a  quadruple as in the beginning of this appendix. 
 Let  $a_s,b_s, s=1, \cdots g, x_1,\cdots, x_k$ be a system of generators of the free group $\pi_1(\pi^{-1}(d_0)\setminus (G\cup E) \cap \pi^{-1}(d_0), \pi^{-1}(d_0) \cap (E'))$ with $x_i$ representing the meridians of the punctures 
  \footnote{i.e. $2g+k={\rm rk} H_1(\pi^{-1}(d_0)\setminus (G\cup E) \cap  \pi^{-1}(d_0))=
 2g+(G,\pi_L^{-1}(d_0))$ where $g$ is the genus of smooth member of the pencil $L$. Explicitly (for appropriate orientation) one has $\tilde x=x_1\cdots x_k \prod_{s=1}^g[a_s,b_s]$ 
 where $\prod_1^g [a_s,b_s]$ is the product of the commutators representing the attaching map of the 2-cell of closed surface of the generators 
 $a_s,b_s, s=1, \cdots g$ of the fundamental group of closed surface $\pi_L^{-1}(d_0)$. For $g=0$ one has $\tilde x=x_1\cdots x_k$.}, 
 and let $\tilde x$ be the word in generators $x_1,\cdots x_k$ and generators of $\pi_1(\pi_L^{-1}(d_0))$ representing the loop which is the  fiber of the normal circle bundle over
 $E$ containing the point  $\widetilde {E} \cap \pi^{-1}(d_0)$.
   Let 
  $$\beta: \pi_1(\PP^1\setminus Crit \cup \bar B' \cup \bar B'') \rightarrow Mod(\pi^{-1}(d_0) \setminus T(E) \cap\pi_L^{-1}(d_0)
  ,\pi^{-1}(d_0)\cap G,\partial T(E\cap\pi_L^{-1}(d_0)))
   $$ be geometric monodromy (cf. (\ref{geommonodromyformula}) and remarks that follow)   
   and let $\gamma_j, j=1,\cdots, N$ be generators of the free group 
  $\pi_1(\CC\setminus Crit  \cup \bar B'',d_0)$ of the complement in $\CC=\PP^1\setminus \bar B'$
    
 In this setting 
one has the following presentations for the fundamental groups of the complements: 

 \begin{equation}\label{redsection222}({\rm A}): \ \ \ \pi_1(\tilde X\setminus G \cup E  \cup \pi_L^{-1}(Crit \cup \bar B' \cup \bar B''), \tilde E\cap \pi_L^{-1}(d_0))=
 \end{equation}
 $$
 \{\gamma_i,x_j,a_s,b_s, i=1,\cdots N, N={\rm Card} (Crit \cup \bar B''),j=1\cdots k, s=1,\cdots g, \vert \beta (\gamma_i)x_j=\gamma^{-1}_ix_j\gamma_i,$$
 $$
 \beta (\gamma_i)a_s=\gamma^{-1}_ia_s\gamma_i,\beta (\gamma_i)b_s=\gamma^{-1}b_s\gamma_i.
 \}
$$
(complement to reducible curve which is a union of a curve, section $E$ of $\pi_L$ at infinity, 
non-generic fibers of the pencil, the fibers at infinity i.e. the fibers over $\bar B' \in \PP^1$ 
and smooth fibers over $\bar B''$) . The same presentation is valid in the case when $G=\emptyset$ i.e. in the case of the complement to a union of 
(singular and smooth) fibers of $\pi_L$ and the section $E$.  

\begin{equation}\label{redsection}({\rm B}):  \ \ \ \ \ \pi_1(\tilde X\setminus G \cup E \cup \pi_L^{-1}(Crit \cup \bar B''),\tilde E \cap \pi_L^{-1}(d_0))=  
\end{equation} 
$$\{\gamma_i,x_j,a_s,b_s, i=1,\cdots N,j=1\cdots k, s=1\cdots g,  \vert \beta_i(\gamma_i)x_j=\gamma^{-1}_ix_j\gamma_i, $$
$$ \ \  \beta (\gamma_i)a_s=\gamma^{-1}_ia_s\gamma_i,\ \ \beta (\gamma_i)b_s=\gamma^{-1}b_s\gamma_i. \ \ 
\gamma_1\cdots \gamma_N=\tilde x^l=1, 
l= (E \cdot E) 
\}
$$
(complement to a reducible curve which is a union of a curve, non-generic and smooth fibers of the pencil 
and $E$) 

 \begin{equation}\label{redaffine} ({\rm C}): \ \ \ \
 \pi_1(\tilde X\setminus G \cup \pi^{-1}(Crit \cup \bar B'\cup \bar B''), \tilde E \cap \pi^{-1}(d_0))
 \end{equation}
 $$
 \{\gamma_i,x_j,a_s,b_s,i=1,\cdots N,j=1\cdots k, s=1,\cdots, g \vert \beta_i(\gamma_i)x_j=\gamma^{-1}_ix_j\gamma_i, $$
$$
\beta (\gamma_i)a_s=\gamma^{-1}_ia_s\gamma_i,\ \ \beta (\gamma_i)b_s=\gamma^{-1}b_s\gamma_i. \ \ \  x_1\cdots x_k \prod_{s=1}^g[a_s,b_s]=1\} 
$$
(complement to the reducible curve which is a union of a curve, non-generic and smooth fibers of the pencil and the fiber at infinity.)
\begin{equation}\label{redproj}  ({\rm D}): \ \ \ \ \ 
\pi_1(\tilde X\setminus G \cup \pi^{-1}(Crit \cup \bar B''), \tilde E \cap \pi^{-1}(d_0))=  
\end{equation} 
$$
\{\gamma_i,x_ja_s,b_s, ,i=1,\cdots N,j=1\cdots k, s=1,\cdots, g \vert \beta_i(\gamma_i)x_j=\gamma^{-1}_ix_j\gamma_i, $$
$$ \beta (\gamma_i)a_s=\gamma^{-1}_ia_s\gamma_i,\ \ \beta (\gamma_i)b_s=\gamma^{-1}b_s\gamma_i. 
 \ \  \gamma_1\cdots \gamma_N=1, \tilde x=1 \}
$$
(complement to the union of a curve, non-generic and smooth fibers of the pencil)
\begin{equation}\label{classicalcase} ({\rm E}) \ \ \ \    \pi_1(\tilde X\setminus G \cup \pi_L^{-1}(\bar B''), \tilde )=\{x_j,a_s,b_s, j=1\cdots k, s=1,\cdots g  \vert \beta_i(\gamma_i)x_j=x_j,$$
$$ \beta (\gamma_i)a_s=a_s\ \ \beta (\gamma_i)b_s=b_s, i=1,\cdots N, 
  \ \  x_1\cdots x_k \prod_{s=1}^g[a_s,b_s] =1\}
\end{equation}
(complement to a curve and smooth fibers)

In the case when some of the singular members of the pencil $L$  are reducible normal crossing divisors,  
the relations corresponding to filling in such a member are as follows.  Let $r_i>1$ denotes the number of irreducible components 
of the $i$-th such singular member.
Let $\Delta_i$ is a disk bounding the loop representing a generator of $\pi_1(\PP^1\setminus Crit \cup \bar B'\cup \bar B'')$ corresponding to the $i$-th reducible fiber of $\vert L \vert $, let 
$\mu^i_j \in \pi_1(\pi_L^{-1}(\partial \Delta_i)), j=1,\cdots r_i$,
are the meridians of irreducible components of the $i$-th singular member and let  
$\epsilon^i_j=({\mu^i_j})^{-1}\mu^i_{j+1}$. \footnote{For simplicity the subscripts in the expression for $\epsilon^j$ 
selected for the case when the dual graph of the singular member is a chain, which is the only instance we use in this paper. 
The same expressions, with indexing depending on the ordering of irreducible components, take place in the case when the dual graph is a tree for any of its vertices}  Then $\epsilon^i_j$ defines the conjugacy class in $\pi_1(\pi_L^{-1}(d_0),\tilde E\cap \pi_1^{-1}(d_0))$
containing a vanishing cycle i.e. a free loop in $\pi_L^{-1}(d_0)$ which bounds a disk in $\pi_L^{-1}(\Delta)$.
The relations corresponding to filling in such a singular member in (\ref{classicalcase}) can be taken as follows: 
\begin{equation}\label{killingmeridians}
               \mu^i_1= \epsilon^i_j=1 \ \ \ j=1, \cdots r_i-1.
\end{equation}    
 
\begin{proof} Let $\gamma_i, i=1,\cdots N$ be a collection of nonintersecting loops in $\PP^1\setminus (Crit \cup \bar B' \cup \bar B'')$
based at $d_0$, each enclosing exactly one point $d_i$ from the set $Crit$. It is ordered via the counterclockwise order around $d_0$ of segments $\gamma_i \cap D_0$ where $D_0$ is a small disk centered at $d_0$ (giving a good ordered system of generators on $\pi_1(\PP^1\setminus D\cup \bar B')$
 cf. \cite{moishezon}). For each $\gamma_i$, $\pi^{-1}(\gamma_i)\setminus G\cup E \cup \pi^{-1}(d_i))$ is diffeomorphic to 
 a locally trivial fibration over circle with fiber being a Riemann surface diffeomorphic to $\pi^{-1}(d_0) \setminus \pi^{-1}(d_0) \cap s(\PP^1)$ and hence 
 the homotopy exact sequence gives a presentation as an HNN extension of free group by a free group: 
 \begin{equation}\label{hnn}
 \pi_1(\pi^{-1}(\gamma_i)\setminus G\cup E \cup \pi^{-1}(d_i)),\tilde E \cap \pi^{-1}(d_0))= \{s=1,\cdots, g, i=1,\cdots N, 
 \end{equation}
 $$
\gamma_i,x_1,\cdots x_k, a_s,b_s, s=1  
 \vert
  \beta_i(\gamma_i)x_j=\gamma^{-1}_ix_j\gamma_i, 
 \beta (\gamma_i)a_s=a_s\ \ \beta (\gamma_i)b_s=b_s.  
 \}
 $$
 The retraction of $(\PP^1\setminus (Crit \cup \bar B'' \cup \bar B'),d_0)$ onto 
$\cup_1^N \gamma_i$ induces the homotopy equivalence between $\tilde X\setminus G\cup \pi^{-1}(Crit \cup B' \cup \bar B'') \cup E$ and 
 $\pi^{-1}(\cup_1^N \gamma_i)$ and Seifert-van Kampen theorem for the fundamental group of a union of spaces shows that the last presentation 
 of the fundamental group of preimage of single loop $\gamma_i$ yields presentation as in classical proofs of the Zaiski-van Kampen relations (\ref{redsection222}).

To show (\ref{redsection}) we apply Seifert-van Kampen theorem to the decomposition: 
\footnote{Note that the map $\iota: \pi_1(\partial T(\pi_L^{-1}(\bar B')) \rightarrow \pi_1(T(\pi_L^{-1}(\bar B'))$ induced by injection is surjective 
with the kernel being the normal subgroup generated by the meridian in $\pi_1(\partial T(\pi_L^{-1}(\bar B'))$ 
Indeed, $T(\pi^{-1}(p))$, where $\pi^{-1}(p)$ possibly singular fiber, can be retracted, due to conical structure of isolated singularities, 
onto  $T(\pi^{-1}(p))\setminus \Delta$, where 
$\Delta=\Delta_1^v \times \Delta_1^h$ is the polydisk centered at the singular point of $\pi_L^{-1}(p)$ and boundaries of the 1-disks 
being respectively the longitude and meridian of the link of the singularity.
Using the amalgamated product decomposition corresponding to the union: 
$\partial T(\pi_L^{-1}(p))=\{\partial T(\pi_L^{-1}(p)) \setminus (\Delta \cap \partial T(\pi_L^{-1}(p)))\} \bigcup_{\Delta \cap \partial T(\pi_L^{-1}(p))} \{\Delta 
\cap \partial T(\pi_L^{-1}(p))\}$ and the amalgamated product decomposition corresponding to similar union via  
removing $\Delta$ from $T(\pi_L^{-1}(p))$ we obtain the 
assertion about structure of the homomorphism $\iota$ since it is induced by inclusion of factors of amalgamated product, where the surjectivity is immediate 
since the kernel being the normal subgroup generated by the meridian $\partial \Delta^h_1$. Similar argument applies when $\pi_L^{-1}(\bar B')$ has more than one singular point.}
:
$$\tilde X\setminus G \cup \pi_L^{-1}(Crit \cup \bar B'')\cup E=
\tilde X\setminus G \cup \pi_L^{-1}(Crit \cup \bar B' \cup \bar B'')\cup E \bigcup_{\partial T(\pi_L^{-1}(\bar B'))}
 T(\pi_L^{-1}(\bar B'))
$$
which implies that $\pi_1(\tilde X\setminus G \cup \pi_L^{-1}(Crit \cup \bar B') \cup E)$ is the quotient of the group given by presentation
(\ref{redsection222}) by the normal subgroup generated by the meridian of the fiber at infinity $\pi_L^{-1}(\bar B')$. The latter is conjugate 
to the loop identified via gluing map of the circle bundle associated to with the normal bundle to $E$ with 
the $l$-th power of the loop given by the fiber of this bundle i.e. $\tilde x$. This shows (\ref{redsection}).

The isomorphism (\ref{redproj}) also follows from Seifert-van Kampen theorem applied to decomposition
$$\tilde X\setminus G \cup \pi^{-1}(Crit \cup \bar B'') =X\setminus G \cup \pi^{-1}(Crit \cup \bar B'') \cup E \bigcup_{\partial T(E) \cap X\setminus G \cup \pi^{-1}(Crit 
\cup \bar B'')}T(E)
$$
showing that the group (\ref{redproj}) is the quotient of group (\ref{redsection}) by the normal closure of $\tilde x$. 

Finally (\ref{classicalcase}) follows from (\ref{redproj}) by applying Seifert-van Kampen theorem to the union of a regular neighborhood of  irreducible components of $\pi^{-1}(Crit \cup \bar B'')$ i.e.
 by adding relations in the normal closure of the meridians of these irreducible components i.e. $\gamma_i=1$. 
 
 In the case when $\vert L \vert$ has reducible fiber as in the theorem, one has to apply Seifert-van Kampen theorem to union of  the regular neighborhoods 
 of all irreducible components since meridians $\mu^i_j$ of different components need not to be conjugate in $\pi_1(\pi_L^{-1}(\partial(\Delta)))$ where $\Delta$ is the disk
 bounded by the loop $\gamma_i$ in the base of the pencil around the point corresponding to the reducible fiber. The filling in relation $\gamma_i=1$ in the 
 fundamental group of base of $\vert L \vert$
 corresponds to relations $\mu^i_j, j=1,\cdots r_i$.  Let $\mu^i_1$ is the meridian of the component intersecting section $E$. 
 It is not hard to check that if $\mu^i_j,\mu^i_{j+1}$ are meridians of intersecting components then $\mu^i_j(\mu^i_{j+1})^{-1}$ in
  $\pi_1^{-1}(\partial \pi_L^{-1}(\gamma_i))$ is represented by  
 the vanishing cycle in $\pi_1^{-1}(d_0)$. Indeed, the map 
 $\pi_{L,*}: \pi_1(\pi_L^{-1}(\partial \Delta_i)) \rightarrow \pi_1(\partial \Delta_i)$ takes all meridians $\mu^i_j$ to generator of $\pi_1(\partial \Delta_i)$ 
 and hence $((\mu^i_j)^{-1}) \mu^i_{j+1}, j=1,\cdot r_i-1$ is in the kernel of homomorphism $\pi_{L,*}$ of the fundamental groups which is $\pi_1(\pi^{-1}(d_0))$ (cf. (\ref{hnn})). Since $\mu^i_j$ are in the kernel of the map
$ \pi_1(\pi_L^{-1}(\partial \Delta_i)) \rightarrow \pi_1(\pi_L^{-1}(\Delta_i))$ induced by embedding so are $\epsilon^i_j, j=1,\cdots r_i-1$. Since ${\rm Ker} \pi_1(\pi_L^{-1}(d_0) \rightarrow \pi_1(\pi_L^{-1}(\Delta_i))$ is generated by $r_i-1$ conjugacy classes of the vanishing cycles the claim follows.
 Hence the set of relations $\mu^i_1=1, \epsilon^i_j=1,j=1,\cdots r_i-1$ is equivalent to the set of relations 
 $\mu^i_j=1, j=1, \cdots r_i$. This shows the assertion of the theorem in case of reducible fibers. 
\end{proof} 

\end{thm}

\end{document}